\documentclass[11pt, twoside]{article}

\usepackage{amsmath}
\usepackage{amssymb}
\usepackage{titletoc}
\usepackage{mathrsfs}
\usepackage{amsthm}

\usepackage{indentfirst}
\usepackage{color}
\usepackage{txfonts}
\usepackage{bbm}
\usepackage{tikz}
\usetikzlibrary{decorations.pathreplacing}

\allowdisplaybreaks

\def\red{\color{red}}

\def\bint{{\ifinner\rlap{\bf\kern.30em--}
\int\else\rlap{\bf\kern.35em--}\int\fi}\ignorespaces}

\def\sbint{{\ifinner\rlap{\bf\kern.32em--}
\hspace{0.078cm}\int\else\rlap{\bf\kern.45em--}\int\fi}\ignorespaces}

\def\rr{\mathbb{R}}
\def\rn{\mathbb{R}^n}

\def\nn{\mathbb{N}}
\def\zz{\mathbb{Z}}

\def\lz{\lambda}

\def\dz{\delta}

\def\gz{{\gamma}}

\def\ls{\lesssim}

\def\fz{\infty}
\def\az{\alpha}

\def\cb{{\mathcal B}}

\def\cq{{\mathcal Q}}

\def\BMO{\mathrm{BMO}}

\def\r{\right}
\def\lf{\left}

\def\gfz{\genfrac{}{}{0pt}{}}

\def\noz{{\nonumber}}

\def\r{\right}
\def\lf{\left}
\def\gfz{\genfrac{}{}{0pt}{}}

\def\loc{{\mathop\mathrm{\,loc\,}}}

\def\BMO{{\mathop\mathrm{\,BMO\,}}}

\def\eqref#1{(\ref{#1})}

\def\blue{\color{blue}}

\theoremstyle{definition}

\newtheorem{theorem}{Theorem}[section]
\newtheorem{proposition}[theorem]{Proposition}
\newtheorem{lemma}[theorem]{Lemma}
\newtheorem{corollary}[theorem]{Corollary}

\newtheorem{definition}[theorem]{Definition}

\newtheorem{remark}[theorem]{Remark}

\numberwithin{equation}{section}

\begin{document}
	
\title{\bf\Large John--Nirenberg-$Q$ Spaces via Congruent Cubes
\footnotetext{\hspace{-0.35cm} 2020 {\it
			Mathematics Subject Classification}. Primary 42B35; Secondary 46E30, 46E35. \endgraf
{\it Key words and phrases}. John--Nirenberg space, congruent cube,
$Q$ space, (fractional) Sobolev space, mean oscillation,
dyadic cube, composition operator.
		\endgraf
This project is partially supported by the National Natural Science Foundation of China
(Grant Nos. 12122102 and 11871100)
and the National Key Research
and Development Program of China
(Grant No. 2020YFA0712900).
}
}
\date{ }
\author{Jin Tao, Zhenyu Yang and
Wen Yuan\footnote{Corresponding author, E-mail: wenyuan@bnu.edu.cn/
{\red December 1, 2021}/Final version.}
}
\maketitle

\vspace{-0.7cm}

\begin{center}
\begin{minipage}{13cm}
{\small {\bf Abstract}\quad
To shed some light on the John--Nirenberg space,
the authors in this article introduce
the John--Nirenberg-$Q$ space via congruent cubes,
$JNQ^\alpha_{p,q}(\mathbb{R}^n)$,
which when $p=\fz$ and $q=2$ coincides with
the space $Q_\alpha(\mathbb{R}^n)$ introduced by
Ess\'en et al. in [Indiana Univ. Math. J. 49 (2000), 575-615].
Moreover, the authors show that, for some particular indices,
$JNQ^\alpha_{p,q}(\mathbb{R}^n)$
coincides with the congruent John--Nirenberg space,
or that the (fractional) Sobolev space is continuously embedded into
$JNQ^\alpha_{p,q}(\mathbb{R}^n)$.
Furthermore, the authors characterize
$JNQ^\alpha_{p,q}(\mathbb{R}^n)$ via mean oscillations,
and then use this characterization to study
the dyadic counterparts.
Also, the authors obtain some properties of composition operators
on such spaces.
The main novelties of this article are twofold:
establishing a general equivalence principle for
a kind of `almost increasing' set function introduced
in this article,
and using the fine geometrical properties of
dyadic cubes to properly classify any collection of cubes with
pairwise disjoint interiors and equal edge length.
}
\end{minipage}
\end{center}

\vspace{0.1cm}

\tableofcontents

\vspace{0.1cm}

\section{Introduction}

Throughout the whole article, a \emph{measurable function $f$}
means that $f$ is Lebesgue measurable and $|f|<\fz$ almost everywhere;
a \emph{cube $Q$} means that it has finite edge length and
all its edges parallel to the coordinate axes,
but $Q$ is not necessary to be open or closed.
Recall that the \emph{Lebesgue space} $L^q(\rn)$
with $q\in[1,\fz]$ is defined to be the set of
all measurable functions $f$ on $\rn$ such that
$$\|f\|_{L^q(\rn)}:=
\begin{cases}
	\displaystyle{\lf[\int_{\rn}|f(x)|^q\,dx\right]^\frac1q},
	&\quad q\in [1,\fz),\\
	\displaystyle{\mathop{\mathrm{ess\,sup}}_{x\in\rn}|f(x)|},
	&\quad q=\fz
\end{cases}$$
is finite.
In what follows, we use $\mathbf{1}_E$ to denote
the \emph{characteristic function} of any set $E\subset\rn$,
and $L_{\loc}^q(\rn)$ to denote
the set of all $f\in L_{\loc}^1(\rn)$ such that
$f{\mathbf 1}_E\in L^q(\rn)$
for any bounded measurable set $E\subset\rn$.

The following John--Nirenberg space via congruent cubes
(for short, the congruent John--Nirenberg space)
was introduced by Jia et al. in \cite{jtyyz21}
to shed some light on the mysterious space $JN_p$ and the space $\cb$,
which are introduced, respectively, by John and Nirenberg \cite{jn61}
and Bourgain et al. \cite{bbm15}.
In what follows, for any $\ell\in(0,\fz)$, let
$\Pi_\ell$ denote the class of all collections $\{Q_j\}_j$ of subcubes of $\rn$
with pairwise disjoint interiors and equal edge length $\ell$.
Moreover, for any $f\in L_{\loc}^1(\rn)$ and
any bounded measurable set $E\subset\rn$ with positive measure,
let
\begin{align*}
f_E:=\fint_E f(y)\,dy:=\frac{1}{|E|}\int_E f(y)\,dy.
\end{align*}

\begin{definition}\label{def-JNconpqa}
Let $p,\ q\in[1,\fz)$.
The \emph{congruent John--Nirenberg space $JN^{\rm con}_{p,q}(\rn)$}
is defined to be the set of all $f\in L_{\loc}^1(\rn)$ such that
\begin{align*}
\|f\|_{JN^{\rm con}_{p,q}(\rn)}
:=\sup_{\gfz{\ell\in(0,\fz)}{\{Q_j \}_{j}\in\Pi_\ell}}
\lf\{\sum_{j}\lf|Q_{j}\r|\lf[{\rm MO}_{f,q}(Q_j)\r]^p\r\}^{\frac{1}{p}}
<\fz
\end{align*}
with the {\em mean oscillation}
\begin{align}\label{MOfq}
{\rm MO}_{f,q}(Q_j):=
\lf[\fint_{Q_j}\lf|f(x)-f_{Q_j}\r|^{q}\,dx\r]^{\frac 1q}
\end{align}
for any $j$.
Moreover, let $JN_{p}^{\mathrm{con}}(\rn)
:=JN_{p,1}^{\mathrm{con}}(\rn)$.
\end{definition}
Very recently, Jia et al. \cite{jtyyzS1,jtyyzS2,jyyzLP} further showed that
several important operators
(such as the Hardy--Littlewood maximal operator, Calder\'on--Zygmund operators,
fractional integrals, and Littlewood--Paley operators) are bounded
on congruent John--Nirenberg spaces.
Thus, it is meaningful to study and reveal more properties of
the congruent John--Nirenberg space.

Another well-known space appeared in John and Nirenberg \cite{jn61}
is $\BMO(\rn)$,
the space containing functions of bounded mean oscillation,
which can be regarded as the limit
space of $JN^{\rm con}_{p,q}(\rn)$ as $p\to\fz$;
see \cite[Proposition 2.21]{jtyyz21} and also
\cite[Proposition 2.6]{tyyNA}.
The space $\BMO(\rn)$ has wide applications
in harmonic analysis and partial differential equations;
see, for instance, \cite{bdlwy19,cdlsy21,cdlsy17,
dllw19,dlsvwy21,dlwy21,dy05,lw17,txyyJFAA,ycyy20}.
In particular, we refer the reader to \cite{tyyM} for a systematic survey
on function spaces of John--Nirenberg type.
Later, Ess\'en et al. \cite{ejpx00} introduced $Q$ spaces
which generalize the space $\BMO(\rn)$.
To be precise, the {\em space $Q_\az(\rn)$} with $\az\in\rr$
is defined to be the set of all measurable functions $f$ on $\rn$
such that
$$\|f\|_{Q_{\az}(\rn)}:=\sup_{{\rm cube\ }Q}
\lf[|Q|^{\frac{2\az}n-1}\int_{Q}\int_{Q}
\frac{|f(x)-f(y)|^2}{|x-y|^{n+2\az}}\,dy\,dx\r]^\frac12$$
is finite,
where the supremum is taken over all cubes $Q$ of $\rn$.
Then Ess\'en et al. in \cite[Theorem 2.3(iii)]{ejpx00} showed that,
for any $\az\in(-\fz,0)$,
$Q_\az(\rn)=\BMO(\rn)$ with equivalent norms.
Later, based on the open problems posed in \cite[Section 8]{ejpx00},
such $Q$ spaces have attracted a lot of attention;
we refer the reader to the recent monograph \cite{x19}
for the elaborate developments of $Q$ spaces.
Therefore, it is natural to consider the corresponding
$Q$ spaces of congruent John--Nirenberg spaces,
which is the main motivation of this article.

In this article,
we introduce the following John--Nirenberg-$Q$ space via congruent cubes
(for short, the $JNQ$ space),
$JNQ^\alpha_{p,q}(\mathbb{R}^n)$,
which when $p=\fz$ and $q=2$ coincides with
the space $Q_\alpha(\mathbb{R}^n)$.
Moreover, we show that, for some particular indices,
the $JNQ^\alpha_{p,q}(\mathbb{R}^n)$
coincides with the congruent John--Nirenberg space,
or that the (fractional) Sobolev space is continuously embedded into
$JNQ^\alpha_{p,q}(\mathbb{R}^n)$.
Furthermore, we characterize $JNQ^\alpha_{p,q}(\mathbb{R}^n)$
via mean oscillations,
and then use this characterization to study
the dyadic counterparts of $JNQ^\alpha_{p,q}(\mathbb{R}^n)$.
Also, we obtain some properties of composition operators
on such spaces.
The main novelties of this article are twofold:
establishing a general equivalence principle for
a kind of `almost increasing' set function introduced
in this article,
and using the fine geometrical properties of
dyadic cubes to properly classify any collection of cubes with
pairwise disjoint interiors and equal edge length.
\begin{definition}\label{def-JNQpqa}
Let $p,\ q\in[1,\fz)$ and $\az\in\rr$.
The \emph{space $JNQ^{\az}_{p,q}(\rn)$} is defined to be the set of
all measurable functions $f$ on $\rn$  such that
\begin{align*}
\|f\|_{JNQ^\az_{p,q}(\rn)}:=
\sup_{\gfz{\ell\in(0,\fz)}{\{Q_j \}_{j}\in\Pi_\ell}}
\lf\{\sum_{j}\lf|Q_{j}\r|
\lf[\Phi_{f,q,\az}(Q_j) \r]^p \r\}^{\frac{1}{p}}<\fz,
\end{align*}
here and thereafter, for any cube $Q$ of $\rn$,
\begin{align}\label{Phifqa}
\Phi_{f,q,\az}(Q):=
\lf[|Q|^{\frac{q\az}n-1}\int_{Q}\int_{Q}
\frac{|f(x)-f(y)|^q}{|x-y|^{n+q\az}}\,dy\,dx \r]^\frac{1}q.
\end{align}
\end{definition}
\begin{remark}\label{rem-moC}
\begin{enumerate}
\item[\rm (i)]
Since $\|f\|_{JNQ^\az_{p,q}(\rn)}=0$ if and only if
$f$ is a constant almost everywhere,
we regard $JNQ^\az_{p,q}(\rn)$ as a function space of modulo constants.
Thus, throughout the whole article, we simply write
$\{{\rm a.\,e.\ constant}\}$ by $\{0\}$.
	
\item[\rm (ii)]
Let $q\in[1,\fz)$ and $\az\in\rr$.
The \emph{space $JNQ^{\az}_{\fz,q}(\rn)$}
can be automatically defined to be the set of
all measurable functions $f$ on $\rn$  such that
$$\|f\|_{JNQ^\az_{\fz,q}(\rn)}:=\sup_{{\rm cube\ }Q}
\Phi_{f,q,\az}(Q)<\fz,$$
where $\Phi_{f,q,\az}(Q)$ is as in \eqref{Phifqa},
and the supremum is taken over all cubes $Q$ of $\rn$.
Then
$$JNQ^{\az}_{\fz,2}(\rn)=Q_\az(\rn)$$
with equal norms,
and hence $Q_\az(\rn)$ can be regarded as the limit space of
$JNQ^{\az}_{p,2}(\rn)$ as $p\to\fz$;
see also Proposition \ref{prop-limit} below.
\end{enumerate}
\end{remark}

The remainder of this article is organized as follows.

Section \ref{sec-rela} is devoted to
revealing the relations between $JNQ$ spaces and
some other function spaces
including congruent John--Nirenberg spaces,
$Q$ spaces, and (fractional) Sobolev spaces.
To be precise, we show $JNQ^\az_{p,q}(\rn)=JN^{\rm con}_{p,q}(\rn)$
for any $\az\in(-\fz,0)$ in Theorem \ref{thm-JNQ123} below.
A main tool used in this proof is the equivalent integral-type
norm (see Corollary \ref{coro-IntJNQ} below).
Indeed, we obtain a more general equivalence principle
in Proposition \ref{prop-integral} below via
introducing the `almost increasing' set function $\Phi$
(see Definition \ref{def-ai} below).
It should be pointed out that
Proposition \ref{prop-integral} may have independent interest
because, as a special case of this proposition with $\Phi$
replaced by the mean oscillation as in \eqref{MOfq},
\cite[Proposition 2.2]{jtyyz21} proves extremely useful when
studying the boundedness of the Hardy--Littlewood maximal operator,
Calder\'on--Zygmund operators, fractional integrals,
and Littlewood--Paley operators on
congruent John--Nirenberg spaces;
see \cite{jtyyzS1,jtyyzS2,jyyzLP} for more details.
Moreover, we show that $Q_\az(\rn)$ can be regarded as the limit space of
$JNQ^{\az}_{p,2}(\rn)$ as $p\to\fz$ in Proposition \ref{prop-limit} below,
and also obtain an extension result over cubes (with finite edge length)
in Proposition \ref{prop-exten} below.
Furthermore, for some non-negative $\az$, we prove that
the (fractional) Sobolev space is continuously embedded into
the $JNQ$ space in Propositions \ref{prop-Wsp} and \ref{prop-W1r} below.
At the end of this section, we sum up these relations
in Theorem \ref{thm-clas} below, which is a complete classification
over $p$, $q\in[1,\fz)$ and $\az\in\rr$.

In Section \ref{sec-MO&Dya}, we first characterize $JNQ$ spaces
via mean oscillations in Theorem \ref{thm-mean-osc} below.
Then we use this characterization to study
the dyadic counterparts of $JNQ$ spaces which
prove to be the intersection of $JNQ$ spaces
and congruent John--Nirenberg spaces;
see Theorem \ref{thm-dyadic} below.
To handle with a collection of cubes in the supremum of $\|\cdot\|_{JNQ^{\az}_{p,q}(\rn)}$
rather than a single cube in the supremum of $\|\cdot\|_{Q_\az(\rn)}$,
we use some fine geometrical properties of dyadic cubes
to properly classify any collection of cubes with
pairwise disjoint interiors and equal edge length;
see Lemma \ref{lem-partition} below.

Section \ref{sec-co} is devoted to investigating the left and the right
composition operators on $JNQ$ spaces.
As an application of Proposition \ref{prop-W1r}
(namely, the Sobolev space is continuously embedded into the $JNQ$ space),
we show that the left composition operator is bounded
on the $JNQ$ space if and only if the corresponding
function belongs to the Lipschitz space;
see Theorem \ref{thm-LC} below.
Moreover, we give a brief discussion about
the right composition operator on the $JNQ$ space;
see Proposition \ref{prop-NoQCM} and
Remark \ref{rem-RCO} below.

Finally, we  make some conventions on notation.
Let $\mathbb{N}:=\{1,\,2,...\}$,
$\mathbb{Z}_+:=\mathbb{N}\cup\{0\}$,
$\mathbb{Z}_+^n:=(\mathbb{Z}_+)^n$,
$\mathbf{0}$ denote the \emph{origin} of the Euclidean space,
and $\nabla b$ the \emph{gradient} of $b$.
For any $\gz:=(\gz_1,\ldots,\gz_n)\in\zz_+^n$, let
$D^\gz:=(\frac\partial{\partial x_1})^{\gz_1}
\cdots(\frac\partial{\partial x_n})^{\gz_n}$.
For any $s\in\nn$, $C^s(\rn)$ denotes the set of
all functions $f$ on $\rn$ whose derivatives
$\{D^\gz f\}_{|\gz|= s}$ are continuous.
In addition, we use $C(\rn)$ to denote the set of
all continuous functions on $\rn$,
and $C^\fz(\rn)$ the set of
all infinitely differentiable functions on $\rn$.
For any $p\in[1,\fz]$, let $p'$ be its \emph{conjugate index},
that is, $p'$ satisfies  $1/p+1/p'=1$.
We use ${\mathbf 1}_E$ to denote the \emph{characteristic function}
of a set $E\subset\rn$,
and $|E|$ the \emph{Lebesgue measure} when $E\subset\rn$ is measurable.
Also, we use $\mathbf{0}$ to denote the \emph{origin} of $\rn$,
and $\sharp A$ the cardinality of the set $A$.
Moreover, for any $z\in\rn$ and $\ell\in(0,\fz)$,
we use $Q(z,\ell)$ to denote the cube centered at
$z$ with the length $\ell$;
for any cube $Q$ of $\rn$ and any $\lz\in(0,\fz)$,
we use $\lz Q$ to denote the cube with edge length $\lz\ell(Q)$
and the same center as $Q$.

\section{Relations with Other Function Spaces}\label{sec-rela}

In this section, we show the relations between $JNQ$ spaces
with congruent John--Nirenberg spaces,
$Q$ spaces, and fractional Sobolev spaces in,
respectively, Subsections \ref{subsec-JN}, \ref{subsec-Q},
and \ref{subsec-Sob}.

\subsection{Relations with Congruent John--Nirenberg Spaces}\label{subsec-JN}

We first observe $JNQ^{-n/2}_{p,q}(\rn)=JN^{\rm con}_{p,q}(\rn)$
via the following lemma.
\begin{lemma}\label{osc-equ}
Let $q\in[1,\fz)$, $E$ be a bounded measurable set
in $\rn$ with positive measure,
and $f$ a measurable function on $\rn$.
\begin{itemize}
\item[\rm (i)]
If $f{\mathbf 1}_E\in L^1(\rn)$, then
\begin{align*}
\lf[\fint_E|f(x)-f_E|^q\,dx\r]^\frac1q
\le \lf[\fint_E\fint_E|f(x)-f(y)|^q\,dy\,dx \r]^\frac1q
\le2\lf[\fint_E|f(x)-f_E|^q\,dx\r]^\frac1q
\end{align*}
and, in particular,
\begin{align*}
2\fint_E|f(x)-f_E|^2\,dx
=\fint_E\fint_E|f(x)-f(y)|^2\,dy\,dx.
\end{align*}

\item[\rm (ii)]
If $f{\mathbf 1}_E\notin L^1(\rn)$, then
$$\fint_E\fint_E|f(x)-f(y)|^q\,dy\,dx=\fz.$$
\end{itemize}
\end{lemma}
\begin{proof}
Let $q$, $E$, and $f$ be as in the present lemma.
We first prove (i).
In this case, we have $f_E<\fz$.
By this, the H\"older inequality, and the Minkowski inequality,
we have
\begin{align*}
&\lf[\fint_E|f(x)-f_E|^q\,dx\r]^\frac1q\\
&\quad\le\lf\{\fint_E\lf[\fint_E|f(x)-f(y)|\,dy\r]^q\,dx \r\}^\frac1q
\le\lf[\fint_E\fint_E|f(x)-f(y)|^q\,dy\,dx \r]^\frac1q\\
&\quad\le\lf[\fint_E\fint_E|f(x)-f_E|^q\,dy\,dx \r]^\frac1q
+\lf[\fint_E\fint_E|f_E-f(y)|^q\,dy\,dx \r]^\frac1q\\
&\quad=2\lf[\fint_E|f(x)-f_E|^q\,dx\r]^\frac1q.
\end{align*}
Moreover, for the case $q=2$, we have
\begin{align*}
&\fint_E|f(x)-f_E|^2\,dx\\
&\quad=\fint_E[f(x)-f_E]\overline{[f(x)-f_E]}\,dx
=\fint_E f(x)\overline{f(x)}\,dx
-f_E\overline{f_E}
\end{align*}
and
\begin{align*}
&\fint_E\fint_E|f(x)-f(y)|^2\,dy\,dx\\
&\quad=\fint_E\fint_E [f(x)-f(y)]\overline{[f(x)-f(y)]}\,dy\,dx
=2\lf[\fint_E f(x)\overline{f(x)}\,dx
-f_E\overline{f_E}\r].
\end{align*}
Thus,
$$2\fint_E|f(x)-f_E|^2\,dx
=\fint_E\fint_E|f(x)-f(y)|^2\,dy\,dx,$$
which completes the proof of (i).

Next, we prove (ii). In this case,
we have $\int_E |f(x)|\,dx=\fz$, and hence there exists
a measurable subset $F\subset E$ such that
$|f(x)|\ge1$ for any $x\in F$.
By this and the H\"older inequality, we conclude that
\begin{align*}
&\lf[\fint_E\fint_E|f(x)-f(y)|^q\,dy\,dx\r]^\frac1q\\
&\quad\ge\fint_E\fint_E|f(x)-f(y)|\,dy\,dx
\ge|E|^{-2}\int_F\int_{E\setminus F}\lf[|f(x)|-|f(y)|\r]\,dy\,dx\\
&\quad\ge|E|^{-2}\int_F\int_{E\setminus F}\lf[|f(x)|-1\r]\,dy\,dx\\
&\quad=|E|^{-2}|E\setminus F|\lf[\int_F |f(x)|\,dx-|F|\r]=\fz,
\end{align*}
where, in the last inequality, we used the observations
$|E|>0$, $|F|\le|E|<\fz$, and
\begin{align*}
\int_F |f(x)|\,dx
&=\int_E |f(x)|\,dx-\int_{E\setminus F} |f(x)|\,dx\\
&\ge \int_E |f(x)|\,dx-|E|=\fz.
\end{align*}
This finishes the proof of (ii) and hence of
Lemma \ref{osc-equ}.
\end{proof}

\begin{remark}\label{JNQpqa=JNQcon}
Let $[\fint_E|f(x)-f_E|^q\,dx]^\frac1q:=\fz$
if $f{\mathbf 1}_E\notin L^1(\rn)$.
Then we can extend the definition of $JN^{\rm con}_{p,q}(\rn)$
to all measurable functions $f$ on $\rn$ such that $\|f\|_{JN^{\rm con}_{p,q}(\rn)}<\fz$.
Therefore, as an immediate consequence of Lemma \ref{osc-equ}, we obtain
$JNQ^{-n/q}_{p,q}(\rn)=JN^{\rm con}_{p,q}(\rn)$
with equivalent norms.
\end{remark}

Motivated by \cite[Proposition 2.2]{jtyyz21}
and \cite[Lemma 5.7]{ejpx00}, we establish an
equivalence principle (namely, Proposition \ref{prop-integral} below)
which shows that,
in the congruent setting (namely, all cubes $\{Q_j\}_j$
have equal edge length), the summation is equivalent to
the integral so long as the integrand $\Phi$ is
`almost increasing'.

\begin{definition}\label{def-ai}
A non-negative set function $\Phi$
on all cubes (or balls) $Q$ is said to be {\em almost increasing}
if there exists a positive constant $C$ such that,
for any cubes (or balls) $Q_1$, $Q_2$ satisfying $Q_1\subset Q_2$
and $|Q_1|\le 2^{-n}|Q_2|$, it holds true that
\begin{align}\label{Q1<Q2}
\Phi(Q_1)\le C \Phi(Q_2).
\end{align}
\end{definition}
\begin{remark}
Two examples of almost increasing set functions:
\begin{itemize}
\item[\rm (i)]
For any cube $Q$,
$${\rm MO}_{f,q}(Q)=
\lf[\fint_{Q}\lf|f(x)-f_{Q}\r|^{q}\,dx\r]^{\frac 1q}$$
as in \eqref{MOfq}
with $f\in L^q_\loc(\rn)$ and $q\in[1,\fz)$;
see \cite[Lemma 2.1]{jtyyz21}.

\item[\rm (ii)]
For any cube $Q$,
$$\Psi_{f,2,\az}(Q)
:=\sum_{k=0}^\fz\sum_{I\in\mathscr{D}_k(Q)}
2^{(2\az-n)k}{\rm MO}_{f,2}(I)$$
with $f\in L^2_\loc(\rn)$ and $\az\in(-\fz,\frac12)$;
see \cite[Lemma 5.7]{ejpx00} and also Lemma \ref{Psi-AI} below.
Here and thereafter, for any $k\in\zz_+:=\{0,1,\dots\}$
and any cube $Q$ of $\rn$,
$\mathscr{D}_k(Q)$ denotes the dyadic cubes
contained in $Q$ of level $k$.
\end{itemize}
\end{remark}

In what follows, for any $z\in\rn$ and $\ell\in(0,\fz)$,
we use $Q(z,\ell)$ to denote the cube centered at
$z$ with the edge length $\ell$;
for any cube $Q$ of $\rn$ and any $\lz\in(0,\fz)$,
we use $\lz Q$ to denote the cube with edge length $\lz\ell(Q)$
and the same center as $Q$;
moreover, we use $\sharp A$ to denote the cardinality of the set $A$.

\begin{proposition}\label{prop-integral}
Let $p\in[1,\fz)$ and $\Phi$ be an almost increasing set function
defined on cubes (or balls) of $\rn$ as in Definition \ref{def-ai}.
Then
\begin{align}\label{BQequ}
&\sup_{\gfz{\ell\in(0,\fz)}{\{Q_j \}_{j}\in\Pi_\ell}}
\lf\{\sum_j |Q_j|\lf[\Phi(Q_j)\r]^p\r\}^\frac1p\notag\\
&\quad\sim \sup_{\ell\in(0,\fz)}\lf\{\int_{\rn}
\lf[\Phi(Q(z,\ell))\r]^p\,dz\r\}^\frac1p
\sim \sup_{\gfz{k\in\zz}{\{Q_j \}_{j}\in\Pi_{2^{k}}}}
\lf\{\sum_j |Q_j|\lf[\Phi(Q_j)\r]^p\r\}^\frac1p
\end{align}
with the positive equivalence constants depending only
on $n$ and $\Phi$.
Moreover, \eqref{BQequ} also holds true with $Q(z,\ell)$
replaced by $B(z,\ell):=\{x\in\rn:\ |x-z|<\ell\}$.
\end{proposition}

\begin{proof}
let $p$ and $\Phi$ be as in the present proposition.
First, we show that
\begin{align}\label{l-i-1}
&\sup_{\ell\in(0,\fz)}\lf\{\int_{\rn}\lf[\Phi(Q(z,\ell))\r]^p\,dz\r\}^\frac1p
\ls\sup_{\gfz{k\in\zz}{\{Q_j \}_{j}\in\Pi_{2^{k}}}}
\lf\{\sum_j |Q_j|\lf[\Phi(Q_j)\r]^p\r\}^\frac1p.
\end{align}
For any $\ell\in(0,\fz)$,
let $m_\ell\in\zz$ satisfy
$\ell\in(2^{-m_\ell-1},2^{-m_\ell}]$,
and, for any $\mathbf{j}\in\mathbb{Z}^n$,
let $Q_{\ell,\mathbf{j}}:=2^{-m_\ell}{\mathbf{j}+[0,2^{-m_\ell})^n}$.
Then, for any $\ell\in(0,\fz)$ and $\mathbf{j}\in\mathbb{Z}^n$,
by some geometrical observations,
we conclude that, for any $z\in Q_{\ell,\mathbf{j}}$,
$$Q(z,\ell)\subset 2Q_{\ell,\mathbf{j}}
\quad{\rm and}
\quad|Q(z,\ell)|\le 2^{-n}|2Q_{\ell,\mathbf{j}}|,$$
which, together with \eqref{Q1<Q2}, further implies that
\begin{align}\label{ai1}
\Phi(Q(z,\ell))\lesssim \Phi(2Q_{\ell,\mathbf{j}}).
\end{align}
Let $V:=\lf\{\mathbf{i}=(i_1,\ldots,i_n):\ i_1,\ldots,i_n\in\{0,1\}\r\}$
and $2\mathbb{Z}^n:=\{2\mathbf{j}:\ \mathbf{j}\in\mathbb{Z}^n\}$.
Then, for any $\mathbf{i}\in V$ and $\ell\in(0,\fz)$,
using some geometrical observations again,
we find that
\begin{align}\label{Pi2ell}
\lf\{2Q_{\ell,(\mathbf{j}+\mathbf{i})}\r\}_{\mathbf{j}\in2\mathbb{Z}^n}
\in\Pi_{2^{-m_\ell+1}}
\end{align}
and
\begin{align}\label{rn=cupQ}
\rn=\bigcup_{\mathbf{i}\in V}\bigcup_{\mathbf{j}\in 2\mathbb{Z}^n}
Q_{\ell,(\mathbf{j}+\mathbf{i})}.
\end{align}
From \eqref{rn=cupQ}, \eqref{ai1}, \eqref{Pi2ell},
and the fact that $\sharp V=2^n$,
it follows that, for any $\ell\in(0,\fz)$,
\begin{align*}
&\int_{\rn}\lf[\Phi(Q(z,\ell))\r]^p\,dz\\
&\quad=\sum_{\mathbf{i}\in V}\sum_{\mathbf{j}\in 2\mathbb{Z}^n}
\int_{Q_{\ell,(\mathbf{j}+\mathbf{i})}}\lf[\Phi(Q(z,\ell))\r]^p\,dz
\ls\sum_{\mathbf{i}\in V}\sum_{\mathbf{j}\in 2\mathbb{Z}^n}
\int_{Q_{\ell,(\mathbf{j}+\mathbf{i})}}
\lf[\Phi(2Q_{\ell,(\mathbf{j}+\mathbf{i})})\r]^p\,dz\\
&\quad\sim\sum_{\mathbf{i}\in V}\sum_{\mathbf{j}\in 2\mathbb{Z}^n}
\lf|2Q_{\ell,(\mathbf{j}+\mathbf{i})}\r|
\lf[\Phi(2Q_{\ell,(\mathbf{j}+\mathbf{i})})\r]^p
\ls\sharp V\sup_{\gfz{k\in\zz}{\{Q_j \}_{j}\in\Pi_{2^{k}}}}
\lf\{\sum_j |Q_j|\lf[\Phi(Q_j)\r]^p\r\}\\
&\quad\sim\sup_{\gfz{k\in\zz}{\{Q_j \}_{j}\in\Pi_{2^{k}}}}
\lf\{\sum_j |Q_j|\lf[\Phi(Q_j)\r]^p\r\}
\end{align*}
with the implicit positive constants independent of $\ell$.
Taking the supremum over $\ell\in(0,\fz)$,
we find that \eqref{l-i-1} holds true.

Next, we prove that
\begin{align}\label{l-i-2}
&\sup_{\gfz{\ell\in(0,\fz)}{\{Q_j \}_{j}\in\Pi_\ell}}
\lf\{\sum_j |Q_j|\lf[\Phi(Q_j)\r]^p\r\}^\frac1p
\ls \sup_{\ell\in(0,\fz)}\lf\{
\int_{\rn}\lf[\Phi(Q(z,\ell))\r]^p\,dz\r\}^\frac1p.
\end{align}
Let $\{Q_j\}_{j}\in\Pi_\ell$ with $\ell\in(0,\infty)$.
Observe that, for any $z\in Q_{j}$,
$$Q_{j}\subset Q(z,2\ell)\quad{\rm and}\quad|Q_{j}|=2^{-n} |Q(z,2\ell)|,$$
which, combined with \eqref{Q1<Q2}, implies that
$$\Phi(Q_{j})\lesssim \Phi(Q(z,2\ell))$$
and hence
\begin{align*}
\sum_j |Q_j|\lf[\Phi(Q_j)\r]^p
&=\sum_{j}\int_{Q_{j}}\lf[\Phi(Q_j)\r]^p\,dz
\ls\sum_{j}\int_{Q_{j}}\lf[\Phi(Q(z,2\ell))\r]^p\,dz\\
&=\int_{\rn}\lf[\Phi(Q(z,2\ell))\r]^p\,dz.
\end{align*}
This implies \eqref{l-i-2}, and hence \eqref{BQequ} holds true.

Finally, since balls and cubes are mutually comparable,
if $Q(z,\ell)$ in \eqref{BQequ} is replaced by $B_{\ell}(z)$,
then the conclusion still holds true via repeating the above argument;
we omit the details here.
This finishes the proof of Proposition \ref{prop-integral}.
\end{proof}

\begin{corollary}\label{coro-IntJNQ}
Let $p,\ q\in[1,\fz)$ and $\az\in\rr$.
Then $f\in JNQ^{\az}_{p,q}(\rn)$ if and only if
$f$ is measurable on $\rn$ and
\begin{align*}
\|f\|_{\widetilde{JNQ}^\az_{p,q}(\rn)}
:=\sup_{r\in(0,\infty)}\lf\{\int_{\rn}\lf[r^{q\az-n}
\int_{B(z,r)}\int_{B(z,r)}\frac{|f(x)-f(y)|^q}{|x-y|^{n+q\az}}\,dx\,dy \r]^{\frac{p}{q}}\,dz\r\}^{\frac{1}{p}}<\fz.
\end{align*}
Moreover, $\|\cdot\|_{JNQ^{\az}_{p,q}(\rn)}\sim
\|\cdot\|_{\widetilde{JNQ}^\az_{p,q}(\rn)}$.
\end{corollary}

\begin{proof}
Let $p,\ q,\ \az,$ and $f$ be as in the present corollary.
For any cube (or ball) $Q$ of $\rn$, let
$\Phi_{f,q,\az}(Q)$ be the same as in \eqref{Phifqa}.
We claim that $\Phi_{f,q,\az}$ is almost increasing as in Definition \ref{def-ai}.
Indeed, for any cubes (or balls) $Q_1$, $Q_2$ satisfying $Q_1\subset Q_2$
and $|Q_1|= C_1|Q_2|$, by the properties of Lebesgue integral, we have
\begin{align*}
\Phi_{f,q,\az}(Q_1)
&=|Q_1|^{\frac{q\az}n-1}
\lf[\int_{Q_1}\int_{Q_1}\frac{|f(x)-f(y)|^q}{|x-y|^{n+q\az}}
\,dx\,dy \r]^{\frac{1}{q}}\\
&\le C_1^{\frac{q\az}n-1}|Q_2|^{\frac{q\az}n-1}
\lf[\int_{Q_2}\int_{Q_2}\frac{|f(x)-f(y)|^q}{|x-y|^{n+q\az}}
\,dx\,dy \r]^{\frac{1}{q}}
=C_1^{\frac{q\az}n-1}\Phi_{f,q,\az}(Q_2),
\end{align*}
and hence $\Phi_{f,q,\az}$ is almost increasing as in Definition \ref{def-ai}.
From this claim and Proposition \ref{prop-integral},
it follows that $\|\cdot\|_{JNQ^{\az}_{p,q}(\rn)}\sim
\|\cdot\|_{\widetilde{JNQ}^\az_{p,q}(\rn)}$,
which completes the proof of Corollary \ref{coro-IntJNQ}.
\end{proof}

The basic properties of $JNQ$ spaces are presented as follows,
which are fine counterparts of the corresponding properties
of $Q$ spaces in \cite[Theorem 2.3]{ejpx00}.
\begin{theorem}\label{thm-JNQ123}
Let $p,\ q\in[1,\fz)$.
\begin{enumerate}
\item[\rm (i)] (Decreasing in $\alpha$) If $-\fz<\az_1\le{\az_2}<\fz$, then
$JNQ^{\az_1}_{p,q}(\rn)\supset JNQ^{\az_2}_{p,q}(\rn)$.

\item[\rm (ii)] (Triviality for large $\az$)
If $\az\in\rr$ satisfies
$\az>n(\frac{1}{q}-\frac{1}{p})$ or $\az\ge 1$,
then $ JNQ^\az_{p,q}(\rn)=\{0\}$.

\item[\rm (iii)] (Triviality for negative $\az$)
If $\az\in(-\fz,0)$, then $JNQ^\az_{p,q}(\rn)=JN^{\rm con}_{p,q}(\rn)$
with equivalent norms.
\end{enumerate}
\end{theorem}
\begin{proof}
Let $p,\ q\in[1,\fz)$.
We first show (i).
Let $\az_1$, ${\az_2}\in\rr$ with $\az_1\le{\az_2}$.
Then, for any $f\in JNQ^{\az_2}_{p,q}(\rn)$
and any ball $B$ of $\rn$, we have
\begin{align*}
&|B|^{\frac{q\az_1}n-1}\int_{B}\int_{B}
\frac{|f(x)-f(y)|^q}{|x-y|^{n+q\az_1}}\,dx\,dy\\
&\quad=|B|^{\frac{q\az_1}n-1}\int_{B}\int_{B}
 \frac{|f(x)-f(y)|^q}{|x-y|^{n+q{\az_2}}}|x-y|^{q(\az_2-\az_1)}\,dx\,dy\\
&\quad\ls |B|^{\frac{q\az}n-1+\frac{q(\az_2-\az_1)}n}\int_{B}\int_{B}
 \frac{|f(x)-f(y)|^q}{|x-y|^{n+q{\az_2}}}\,dx\,dy\\
&\quad\sim |B|^{\frac{q{\az_2}}n-1}\int_{B}\int_{B}
\frac{|f(x)-f(y)|^q}{|x-y|^{n+q{\az_2}}}\,dx\,dy,
\end{align*}
which, together with Corollary \ref{coro-IntJNQ}, further implies that
$$\|f\|_{JNQ^{\az_1}_{p,q}(\rn)}\ls\|f\|_{JNQ^{\az_2}_{p,q}(\rn)}$$
and hence $JNQ^{\az_1}_{p,q}(\rn)\supset JNQ^{\az_2}_{p,q}(\rn)$.
This finishes the proof of (i).

We next show (ii).
Let $f\in JNQ^\az_{p,q}(\rn)$.
By Definition \ref{def-JNQpqa},
we conclude that, for any given cube $Q$ of $\rn$
with edge length $\ell(Q)\in(0,\fz)$,
\begin{align}\label{onecube}
\lf[\int_{Q}\int_{Q}\frac{|f(x)-f(y)|^q}{|x-y|^{n+q\az}}\,dx\,dy\r]^\frac1q
\le |Q|^{\frac{1}{q}-\frac{\az}n-\frac1p}\|f\|_{ JNQ^\az_{p,q}(\rn)}.
\end{align}
Letting $|Q|\to\fz$ in \eqref{onecube}, it follows that
$f$ is a constant almost everywhere on $\rn$ if $\az>n(\frac{1}{q}-\frac1p)$.
On the other hand, if $\az\ge1$, then
\begin{align*}
&\int_{Q}\int_{Q}\frac{|f(x)-f(y)|^q}{|x-y|^{n+q\az}}\,dx\,dy
\gtrsim
\int_{Q}\int_{Q}\frac{|f(x)-f(y)|^q}{|x-y|^{n+q}[\ell(Q)]^{q(\az-1)}}\,dx\,dy,
\end{align*}
which, combined with \eqref{onecube}, further implies that
\begin{align*}
\int_{Q}\int_{Q}\frac{|f(x)-f(y)|^q}{|x-y|^{n+q}}\,dx\,dy
\ls|Q|^{1-\frac{q}n-\frac{q}p}\|f\|_{ JNQ^\az_{p,q}(\rn)}
\end{align*}
and hence
\begin{align*}
&\inf_{(x,y)\in Q\times Q}\lf[\frac{|f(x)-f(y)|}{|x-y|}\r]^q
\int_{Q}\int_{Q}\frac{dx\,dy}{|x-y|^{n}}\\
&\quad\le\int_{Q}\int_{Q}\frac{|f(x)-f(y)|^q}{|x-y|^{n+q}}\,dx\,dy
\ls|Q|^{1-\frac{q}n-\frac{q}p}\|f\|_{ JNQ^\az_{p,q}(\rn)}<\fz.
\end{align*}
From this and the observation
$$\int_{Q}\int_{Q}\frac{dx\,dy}{|x-y|^{n}}=\fz,$$
we deduce that
$$\inf_{(x,y)\in Q\times Q}\lf[\frac{|f(x)-f(y)|}{|x-y|}\r]^q=0$$
for any cube $Q$ of $\rn$,
which further implies that $f$ is a constant almost everywhere on $\rn$.
This finishes the proof of (ii).

Now, we show (iii).
Recall that $JN^{\rm con}_{p,q}(\rn)=JNQ^{-n/q}_{p,q}(\rn)$;
see Remark \ref{JNQpqa=JNQcon}.
Thus, to prove (iii), it suffices to show that
$ JNQ^\az_{p,q}(\rn)=JNQ^{-n/q}_{p,q}(\rn)$ for any $\az\in(-\fz,0)$,
and we consider the following two cases on $\az$.

\emph{Case} 1) $\az\in(-\fz, -\frac{n}q]$.
In this case, by (i), we conclude that
$$JNQ^{-n/q}_{p,q}(\rn)\subset JNQ^\az_{p,q}(\rn).$$
Conversely, let $f\in  JNQ^\az_{p,q}(\rn)$ and
$B$ be a ball of $\rn$ with radius $r_B\in(0,\fz)$.
Observe that, for any $x$, $y\in B$,
\begin{align*}
&\lf|\lf\{z\in B:\ \min\{|x-z|,\,|y-z|\}>2^{-1} r_B\r\} \r|\\
&\quad=\lf|B\setminus\lf[B(x,2^{-1} r_B)\cup B(y,2^{-1} r_B)\r] \r|
\ge\frac12|B|.
\end{align*}
From this and $-q\az-n\ge0$, we deduce that
\begin{align*}
&\int_B \min\lf\{|x-z|^{-q\az-n},\,|y-z|^{-q\az-n}\r\}\,dz\\
&\quad=\int_B \lf[\min\lf\{|x-z|,\,|y-z|\r\}\r]^{-q\az-n}\,dz\\
&\quad\ge \int_{\{z\in B:\ \min\{|x-z|,\,|y-z|\}>2^{-1} r_B\}}
 \lf[\min\lf\{|x-z|,|y-z|\r\}\r]^{-q\az-n}\,dz\\
&\quad\gtrsim |B|^{-\frac{q\az}n-1}\lf|\lf\{z\in B:\
\min\lf\{|x-z|,\,|y-z|\r\}>2^{-1} r_B\r\} \r|
\gtrsim |B|^{-\frac{q\az}n}
\end{align*}
and hence
\begin{align}\label{az<-1/2}
&\fint_B\fint_B|f(x)-f(y)|^q\,dx\,dy \noz\\
&\quad\ls |B|^{\frac{q\az}n-2}\int_B\int_B\int_B
 \min\lf\{|x-z|^{-q\az-n},\,|y-z|^{-q\az-n}\r\}
|f(x)-f(y)|^q\,dx\,dy\,dz\noz\\
&\quad\ls |B|^{\frac{q\az}n-2}\int_B\int_B\int_B
 \min\lf\{|x-z|^{-q\az-n},\,|y-z|^{-q\az-n}\r\}\noz\\
&\qquad\times\lf[|f(x)-f(z)|^q+|f(y)-f(z)|^q\r]\,dx\,dy\,dz\noz\\
&\quad\ls |B|^{\frac{q\az}n-1}
\int_B\int_B\frac{|f(x)-f(z)|^q}{|x-z|^{n+q\az}}\,dx\,dz
+|B|^{\frac{q\az}n-1}
\int_B\int_B\frac{|f(y)-f(z)|^q}{|y-z|^{n+q\az}}\,dy\,dz\noz\\
&\quad\sim |B|^{\frac{q\az}n-1}
\int_B\int_B\frac{|f(x)-f(y)|^q}{|x-y|^{n+q\az}}\,dx\,dy.
\end{align}
Then \eqref{az<-1/2} implies that
$\|f\|_{JNQ^{-n/q}_{p,q}(\rn)}\ls\|f\|_{ JNQ^\az_{p,q}(\rn)}$
and hence $JNQ^\az_{p,q}(\rn)\subset JNQ^{-n/q}_{p,q}(\rn)$.
This shows that $JNQ^\az_{p,q}(\rn)= JNQ^{-n/q}_{p,q}(\rn)$
when $\az\in(-\fz, -\frac{n}q]$.

\emph{Case} 2) $\az\in [-\frac{n}q,0)$.
In this case, by (i), we conclude that
$$JNQ^{-n/q}_{p,q}(\rn)\supset JNQ^\az_{p,q}(\rn).$$
Conversely, let $z\in\rn$, $r\in(0,\fz)$, and $f\in JNQ^{-n/q}_{p,q}(\rn)$.
By the Tonelli theorem, we have
\begin{align*}
&\int_{B(z,r)}\int_{B(z,r)}\frac{|f(x)-f(y)|^q}{|x-y|^{n+q\az}}\,dx\,dy\\
&\quad=\int_{B(z,r)}\int_{B(z-y,r)}\frac{|f(x+y)-f(y)|^q}{|x|^{n+q\az}}\,dx\,dy\\
&\quad\le\int_{B(z,r)}\int_{B(\mathbf{0},2r)}
\frac{|f(x+y)-f(y)|^q}{|x|^{n+q\az}}\,dx\,dy\\
&\quad=\int_{B(\mathbf{0},2r)}\int_{B(z,r)}
\frac{|f(x+y)-f(y)|^q}{|x|^{n+q\az}}\,dy\,dx\\
&\quad\ls \int_{B(\mathbf{0},2r)}\frac1{|x|^{n+q\az}}
\int_{B(z,r)}|f(x+y)-f_{B(z,3r)}|^q\,dy\,dx\\
&\qquad +\int_{B(\mathbf{0},2r)}\frac1{|x|^{n+q\az}}\,dx
\int_{B(z,r)}|f_{B(z,3r)}-f(y)|^q\,dy\\
&\quad\ls\int_{B(\mathbf{0},2r)}\frac1{|x|^{n+q\az}}\,dx
\int_{B(z,3r)}|f(y)-f_{B(z,3r)}|^q\,dy\\
&\quad\sim r^{-q\az}\int_{B(z,3r)}|f(y)-f_{B(z,3r)}|^q\,dy,
\end{align*}
where the implicit positive constants are independent of $r$.
This, together with Corollary \ref{coro-IntJNQ}, shows that
\begin{align*}
\|f\|_{ JNQ^\az_{p,q}(\rn)}
&\sim
\sup_{r\in(0,\infty)}\lf[\int_{\rn}\lf[r^{q\az-n}
\int_{B(z,r)}\int_{B(z,r)}\frac{|f(x)-f(y)|^q}{|x-y|^{n+q\az}}\,dx\,dy \r]^{\frac{p}{q}}\,dz\r]^{\frac{1}{p}}\\
&\ls \sup_{r\in(0,\infty)}\lf[\int_{\rn}\lf[
\fint_{B(z,3r)}|f(y)-f_{B(z,3r)}|^q\,dy \r]^{\frac{p}{q}}\,dz\r]^{\frac{1}{p}}
\sim  \|f\|_{JNQ^{-n/q}_{p,q}(\rn)},
\end{align*}
which further implies that $JNQ^\az_{p,q}(\rn)= JNQ^{-n/q}_{p,q}(\rn)$
when $\az\in [-\frac{n}q,0)$.

Combining Cases 1) and 2), we obtain (iii).
This finishes the proof of Theorem \ref{thm-JNQ123}.
\end{proof}

As an application of Theorem \ref{thm-JNQ123},
we can prove the triviality of both $JN^{\rm con}_{p,q}(\rn)$
and $JN_{p,q}(\rn)$ for any $1\le p<q<\fz$,
the latter of which was introduced by Tao et al. \cite{tyyNA}
(and also by Dafni et al. \cite{dhky18} on cubes)
as an generalization of the John--Nirenberg space $JN_p$.
Recall that the {\em space $JN_{p,q}(\rn)$} is defined to be
the set of all $f\in L^1_\loc(\rn)$ such that
$$\|f\|_{JN_{p,q}(\rn)}:=\sup
\lf\{\sum_{j}\lf|Q_{j}\r|\lf[\fint_{Q_{j}}
\lf|f(x)-f_{Q_j}\r|^{q}\,dx\r]^{\frac{p}{q}}\r\}^{\frac{1}{p}}<\fz,$$
where the supremum is taken over all collections $\{Q_j\}_j$ of
subcubes of $\rn$ with pairwise disjoint interiors
($\{Q_j\}_j$ may have different edge lengths).

\begin{corollary}\label{coro-JN=c}
Let $1\le p<q<\fz$.
Then $JN^{\rm con}_{p,q}(\rn)=\{0\}=JN_{p,q}(\rn)$.
\end{corollary}
\begin{proof}
Since $1\le p<q<\fz$,
it follows that $(-\fz,0)\cap(n(\frac{1}{q}-\frac{1}{p}),\fz)\neq\emptyset$.
By this and (ii) and (iii) of Theorem \ref{thm-JNQ123},
we conclude that $JN^{\rm con}_{p,q}(\rn)=\{0\}$.
From this and the observation that
$$JN_{p,q}(\rn)\subset JN^{\rm con}_{p,q}(\rn)
\quad{\rm and}\quad
\|\cdot\|_{JN^{\rm con}_{p,q}(\rn)}
\le\|\cdot\|_{JN_{p,q}(\rn)}$$
for any $p,\ q\in[1,\fz)$,
we deduce that $JN_{p,q}(\rn)=\{0\}$ as well,
which completes the proof of Corollary \ref{coro-JN=c}.
\end{proof}
\begin{remark}
Corollary \ref{coro-JN=c} partially answers the open question
posed in \cite[Remark 4.2(ii)]{tyyNA}
(see also \cite[Question 1(ii)]{tyyM}) for the case $q>p$,
but it is still \emph{unclear} for the case $q=p$.
\end{remark}

\subsection{Relations with $Q$ Spaces}\label{subsec-Q}

Also, the  $JNQ$ space is closely connected with the $Q$ space.
Indeed, the following proposition shows that $Q_\az(\rn)$ serves as
a limit space of $JNQ^{\az}_{p,2}(\rn)$ when $p\to\fz$,
and hence it is reasonable to define $JNQ^{\az}_{\fz,2}(\rn):=Q_\az(\rn)$
in Remark \ref{rem-moC}(ii).

\begin{proposition}\label{prop-limit}
Let $p\in[1,\fz)$ and $\az\in\rr$. Then,
for any $f\in \bigcup_{r\in [1,\fz)}\bigcap_{p\in[r,\fz)}JNQ^{\az}_{p,2}(\rn)$,
it holds true that
$$\lim_{p\to\fz}\|f\|_{JNQ^{\az}_{p,2}(\rn)}=\|f\|_{Q_\az(\rn)}.$$
\end{proposition}
\begin{proof}
Let $p\in[1,\fz),\ \az\in\rr$, and $f$ be a measurable function on $\rn$.
Then, for any given cube $Q$ of $\rn$, by Definition \ref{def-JNQpqa},
we have
\begin{align*}
\|f\|_{JNQ^{\az}_{p,2}(\rn)}
&\geq|{Q}|^{\frac{1}{p}}
\lf[|{Q}|^{\frac{2\az}n-1}\int_{{Q}}\int_{{Q}}
\frac{|f(x)-f(y)|^2}{|x-y|^{n+2\az}}\,dy\,dx \r]^\frac{1}2.
\end{align*}
Since $|Q|^{\frac{1}{p}}\to 1$ as $p\to\fz$, it follows that
\begin{align*}
\liminf_{p\to\fz}\|f\|_{JNQ^{\az}_{p,2}(\rn)}
\geq
\lf[|{Q}|^{\frac{2\az}n-1}\int_{{Q}}\int_{{Q}}
\frac{|f(x)-f(y)|^2}{|x-y|^{n+2\az}}\,dy\,dx \r]^\frac{1}2.
\end{align*}
Then, by the arbitrariness of $Q$, we find that
$$\liminf_{p\to\fz}\|f\|_{JNQ^{\az}_{p,2}(\rn)}
\geq \|f\|_{Q_\az(\rn)}.$$
On the other hand, let
$f\in \bigcup_{r\in [1,\fz)}\bigcap_{p\in[r,\fz)}JNQ^{\az}_{p,2}(\rn)$.
Then there exists an $r_0\in [1,\fz)$ such that
$f\in JNQ^{\az}_{p,2}(\rn)$ for any $p\in[r_0,\fz)$.
We claim that
\begin{align}\label{sup}
\limsup_{p\to\fz}\|f\|_{JNQ^{\az}_{p,2}(\rn)}\leq \|f\|_{Q_\az(\rn)}.
\end{align}
Indeed, if $\|f\|_{Q_\az(\rn)}=\fz$, then \eqref{sup} trivially holds true.
If $\|f\|_{Q_\az(\rn)}<\fz$, then we may assume, without loss of generality,
that $\|f\|_{Q_\az(\rn)}=1$. Thus, for any cube $Q$ of $\rn$,
\begin{align*}
|Q|^{\frac{2\az}n-1}\int_{Q}\int_{Q}\frac{|f(x)-f(y)|^2}{|x-y|^{n+2\az}}\,dy\,dx
\leq\|f\|_{Q_\az(\rn)}=1,
\end{align*}
and hence, for any $p\in[r_0,\fz)$, we have
\begin{align*}
\|f\|^{p}_{JNQ^{\az}_{p,2}(\rn)}
&=\sup_{\gfz{\ell\in(0,\fz)}{\{Q_j \}_{j}\in\Pi_\ell}}
\lf\{\sum_{j}\lf|Q_{j}\r|
\lf[|Q_j|^{\frac{2\az}n-1}\int_{Q_j}\int_{Q_j}
\frac{|f(x)-f(y)|^2}{|x-y|^{n+2\az}}\,dy\,dx \r]^\frac{p}2 \r\}\\
&\leq \sup_{\gfz{\ell\in(0,\fz)}{\{Q_j \}_{j}\in\Pi_\ell}}
\lf\{\sum_{j}\lf|Q_{j}\r|
\lf[|Q_j|^{\frac{2\az}n-1}\int_{Q_j}\int_{Q_j}
\frac{|f(x)-f(y)|^2}{|x-y|^{n+2\az}}\,dy\,dx \r]^\frac{r_0}2 \r\}\\
&=\|f\|^{r_0}_{JNQ^{\az}_{{r_0},2}(\rn)}<\fz.
\end{align*}
Letting $p\to\fz$, we obtain
\begin{align*}
\limsup_{p\to\fz}\|f\|_{JNQ^{\az}_{p,2}(\rn)}\leq 1= \|f\|_{Q_\az(\rn)}.
\end{align*}
This finishes the proof of \eqref{sup} and hence of Proposition \ref{prop-limit}.
\end{proof}

Next, we show some extension properties of $JNQ$ spaces.
Recall that we can extend a measurable function $f$
via a fundamental invariance principle: for any $(x,t)\in \rr^{n+1}$,
\begin{align}\label{Fxt}
	F(x,t):=f(x).
\end{align}
It is easy to show that $\|F\|_{L^\fz(\rr^{n+1})}=\|f\|_{L^\fz(\rn)}$.
Moreover, by \cite[Theorem 2.6]{ejpx00}, we also have
$$\|F\|_{Q_\az(\rr^{n+1})}\sim\|f\|_{Q_\az(\rn)}$$
for any $\az\in\rr$,
where the implicit equivalence constants are independent of $f$ and $F$.
Correspondingly, as in \cite[Proposition 4.1]{dhky18}
and \cite[Lemma 2.18]{jtyyz21},
we next show that such extension
also holds true for the $JNQ$ space on cubes.
In what follows, for any $p,\ q\in[1,\fz)$, any $\az\in\rr$, any cube $Q_0\subset\rn$,
and any measurable function $f$ on $Q_0$, let
\begin{align*}
\|f\|_{JNQ^{\az}_{p,q}(Q_0)}
:=\sup\lf\{\sum_j |Q_j|\lf[\Phi_{f,q,\az}(Q_j)\r]^p\r\}^\frac1p,
\end{align*}
where $\Phi_{f,q,\az}(Q_j)$ is as in \eqref{Phifqa} with
$Q$ replaced by $Q_j$,
and the supremum is taken over all collections $\{Q_j\}_j$ of subcubes of
$Q_0$ with pairwise disjoint interiors.
Moreover, the \emph{space $JNQ^{\az}_{p,q}(Q_0)$} is defined by setting
$$JNQ^{\az}_{p,q}(Q_0):=\lf\{f {\rm\ is\ measurable\ on\ }Q_0:\
\|f\|_{JNQ^{\az}_{p,q}(Q_0)}<\fz \r\}.$$

\begin{proposition}\label{prop-exten}
Let $p$, $q\in[1,\fz)$, $\az\in\rr$, $Q_0$ be a cube of $\rn$
with edge length $\ell_0$, and $\widetilde{Q}_0:=Q_0\times[t_0,t_0+\ell_0]$
for any given $t_0\in\rr$.
Let $f$ be a measurable function on $Q_0$, and
$F(x,t):=f(x)$ for any $(x,t)\in \widetilde{Q}_0$ .
Then $F$ is a measurable function on $\widetilde{Q}_0$, and
\begin{align}\label{Ff}
F\in JNQ^{\az}_{p,q}(\widetilde{Q}_0)
\Longleftrightarrow
f\in JNQ^{\az}_{p,q}(Q_0).
\end{align}
Moreover,
$$\|F\|_{JNQ^{\az}_{p,q}(\widetilde{Q}_0)}
\sim
\ell_0^{1/p}\|f\|_{JNQ^{\az}_{p,q}(Q_0)}$$
with the positive equivalence constants independent of
$f$, $F$, and $Q_0$.
\end{proposition}

\begin{proof}
Let all the symbols be as in the present proposition.
We claim that, if $\az>-\frac nq$, then,
for any $x,$ $y\in\rn$ with $|x-y|\le\ell$ for some $\ell\in(0,\fz)$,
\begin{align}\label{tu}
\int_{t_0}^{t_0+\ell}\int_{t_0}^{t_0+\ell}
\frac{dt\,du}{|(x-y,t-u)|^{n+1+q\az}}
\sim \ell|x-y|^{-n-q\az}.
\end{align}
Indeed, on one hand, by $\az>-\frac nq$, we have
\begin{align}\label{tuLHS}
&\int_{t_0}^{t_0+\ell}\int_{t_0}^{t_0+\ell}
\frac{dt\,du}{|(x-y,t-u)|^{n+1+q\az}}\noz\\
&\quad\sim\int_{t_0}^{t_0+\ell}\int_{t_0}^{t_0+\ell}
\frac{dt\,du}{|x-y|^{n+1+q\az}+|t-u|^{n+1+q\az}}
\ls \ell\int_{\rr}\frac{ds}{|x-y|^{n+1+q\az}+|s|^{n+1+q\az}}\noz\\
&\quad\ls \ell\int_{0}^{\fz}\frac{ds}{(|x-y|+s)^{n+1+q\az}}
\sim\ell\int_{|x-y|}^{\fz}\frac{ds}{s^{n+1+q\az}}
\sim \ell |x-y|^{-n-q\az}.
\end{align}
On the other hand, from $|x-y|\le\ell$, it follows that
\begin{align}\label{tuRHS}
&\int_{t_0}^{t_0+\ell}\int_{t_0}^{t_0+\ell}
\frac{dt\,du}{|(x-y,t-u)|^{n+1+q\az}}\noz\\
&\quad\sim\int_{t_0}^{t_0+\ell}\int_{t_0}^{t_0+\ell}
\frac{dt\,du}{|x-y|^{n+1+q\az}+|t-u|^{n+1+q\az}}
\gtrsim\int_{t_0}^{t_0+\ell}\int_{t_0}^{t_0+\ell}
\frac{\mathbf{1}_{\{|t-u|<|x-y|\}}(t,u)\,dt\,du}
{|x-y|^{n+1+q\az}+|t-u|^{n+1+q\az}}\noz\\
&\quad\gtrsim\frac{\ell|x-y|}{|x-y|^{n+1+q\az}}
\sim\ell |x-y|^{-n-q\az}.
\end{align}
Combining \eqref{tuLHS} and \eqref{tuRHS},
we obtain \eqref{tu} and hence the above claim holds true.

Now, we show \eqref{Ff}.
By Theorem \ref{thm-JNQ123}(iii),
we only need to prove \eqref{Ff} with $\az\in(-\frac nq,\fz)$.
Let $\{\widetilde{Q}_i\}_i:=\{Q_{i}\times I_{i}\}_i$
be any given collection of subcubes
of $\widetilde{Q}_{0}$ with pairwise disjoint interiors
and equal edge length $\ell_1$.
Then, from the Tonelli theorem and \eqref{tu}, it follows that
\begin{align}\label{f-F}
&\sum_i |\widetilde{Q}_i|
\lf[\ell_1^{q\az-(n+1)}
\int_{\widetilde{Q}_i}\int_{\widetilde{Q}_i}
\frac{|F(x,t)-F(y,u)|^q}{|(x,t)-(y,u)|^{n+1+q\az}}
\,dx\,dt\,dy\,du\r]^p\notag\\
&\quad=
\sum_i |\widetilde{Q}_i|
\lf[\ell_1^{q\az-(n+1)}
\int_{\widetilde{Q}_i}\int_{\widetilde{Q}_i}
\frac{|f(x)-f(y)|^q}{|(x,t)-(y,u)|^{n+1+q\az}}
\,dx\,dt\,dy\,du\r]^p\notag\\
&\quad\sim
\sum_i |{Q}_i|\ell_1
\lf[\ell_1^{q\az-n}
\int_{{Q}_i}\int_{{Q}_i}
\frac{|f(x)-f(y)|^q}{|x-y|^{n+q\az}}
\,dx\,dy\r]^p\notag\\
&\quad\sim
\sum_i |{Q}_i|
\int_{0}^{\ell_0}{\mathbf{1}}_{I_i}(t)\,dt
\lf[\ell_1^{q\az-n}
\int_{{Q}_i}\int_{{Q}_i}
\frac{|f(x)-f(y)|^q}{|x-y|^{n+q\az}}
\,dx\,dy\r]^p\notag\\
&\quad\sim
\int_{0}^{\ell_0}
\sum_{\{i:\ I_i\ni t\}} |{Q}_i|
\lf[\ell_1^{q\az-n}
\int_{{Q}_i}\int_{{Q}_i}
\frac{|f(x)-f(y)|^q}{|x-y|^{n+q\az}}
\,dx\,dy\r]^p
\,dt.
\end{align}
Notice that, for any given $t\in(0,\ell_0)$,
$\{Q_i:\ I_i\ni t\}$ is a collection of subcubes of $Q_0$
with pairwise disjoint interiors and equal edge length $\ell_1$,
which implies that
\begin{align*}
\sum_{\{i:\ I_i\ni t\}} |{Q}_i|
\lf[\ell_1^{q\az-n}
\int_{{Q}_i}\int_{{Q}_i}
\frac{|f(x)-f(y)|^q}{|x-y|^{n+q\az}}
\,dx\,dy\r]^p
\leq
\|f\|^{p}_{JNQ^{\az}_{p,q}(Q_0)}.
\end{align*}
By this and \eqref{f-F}, we obtain
\begin{align*}
&\sum_i |\widetilde{Q}_i|
\lf[\ell_1^{q\az-(n+1)}
\int_{\widetilde{Q}_i}\int_{\widetilde{Q}_i}
\frac{|F(x,t)-F(y,u)|^q}{|(x,t)-(y,u)|^{n+1+q\az}}
\,dx\,dt\,dy\,du\r]^p\noz\\
&\quad\lesssim
\int_{0}^{\ell_0}
\|f\|^{p}_{JNQ^{\az}_{p,q}(Q_0)}
\,dt
\sim\ell_0\|f\|^{p}_{JNQ^{\az}_{p,q}(Q_0)}.
\end{align*}
From this and the the arbitrariness of $\{\widetilde{Q}_i\}_i$,
it follows that
\begin{align}\label{F<f}
\|F\|_{JNQ^{\az}_{p,q}(\widetilde{Q}_0)}
\ls
\ell_0^{1/p}\|f\|_{JNQ^{\az}_{p,q}(Q_0)}.
\end{align}

Conversely, let $\{Q_i\}_{i}$ be any given collection of subcubes
of $Q_0$ with pairwise disjoint interiors and equal edge length $\ell_1$.
Also, for any $j\in\{1,\dots,J\}$ with $J$ being the largest
integer not greater than $\ell_0/\ell_1$, let
$$Q_{i,j}:=Q_j\times[t_0+(j-1)\ell_1,t_0+j\ell_1].$$
Then it is obvious that $\{Q_{i,j}\}_{i,j}$ is a collection of
subcubes of $\widetilde{Q}_0$ with pairwise disjoint interiors
and equal edge length $\ell_1$, and $J\ell_1\ge\ell_0/2$.
By this and \eqref{tu}, we have
\begin{align*}
\|F\|^{p}_{JNQ^{\az}_{p,q}(\widetilde{Q}_0)}
&\geq\sum_{i,j} |Q_{i,j}|
\lf[\ell_1^{q\az-(n+1)}\int_{Q_{i,j}}\int_{Q_{i,j}}
\frac{|F(x,t)-F(y,u)|^q}{|(x,t)-(y,u)|^{n+1+q\az}}
\,dx\,dt\,dy\,dv\r]^p\\
&\sim\sum_{j=1}^J\sum_i|{Q}_{i}|\ell_1
\lf[\ell_1^{q\az-n}
\int_{{Q}_{i}}\int_{{Q}_{i}}
\frac{|f(x)-f(y)|^q}{|x-y|^{n+q\az}}
\,dx\,dy\r]^p\\
&\gtrsim
\frac{\ell_0}{2}
\sum_i|{Q}_{i}|\lf[\ell_1^{q\az-n}
\int_{{Q}_{i}}\int_{{Q}_{i}}
\frac{|f(x)-f(y)|^q}{|x-y|^{n+q\az}}
\,dx\,dy\r]^p.
\end{align*}
From this and the the arbitrariness of $\{Q_i\}_i$,
it follows that
\begin{align*}
\|F\|_{JNQ^{\az}_{p,q}(\widetilde{Q}_0)}
\gtrsim
\ell_0^{1/p}\|f\|_{JNQ^{\az}_{p,q}(Q_0)},
\end{align*}
which, combined with \eqref{F<f},
further shows that
\begin{align*}
\|F\|_{JNQ^{\az}_{p,q}(\widetilde{Q}_0)}
\sim
\ell_0^{1/p}\|f\|_{JNQ^{\az}_{p,q}(Q_0)}.
\end{align*}
This finishes the proof of Proposition \ref{prop-exten}.
\end{proof}

\begin{remark}\label{rem-extRn}
The factor $\ell_0^{1/p}$ in Proposition \ref{prop-exten} indicates that
the invariance principle \eqref{Fxt} is no longer feasible
(unless $p=\fz$)
for the extension from $JNQ^{\az}_{p,q}(\rn)$ to
$JNQ^{\az}_{p,q}(\rr^{n+1})$.
\end{remark}

\subsection{Relations with (Fractional) Sobolev Spaces}\label{subsec-Sob}

Now, we show the relations between $JNQ$ spaces
and (fractional) Sobolev spaces.

Recall that the \emph{fractional Sobolev space}
$W^{s,p}(\rn)$ with $p\in[1,\fz)$ and $s\in\rr$
is defined to be the set of all measurable functions
$f$ on $\rn$ such that
$$
\|f\|_{W^{s,p}(\rn)}:=
\lf[\int_{\rn}\int_{\rn}
\frac{|f(x)-f(y)|^p}{|x-y|^{n+sp}}\,dx\,dy\r]^\frac1p<\fz.
$$
When $s\in(0,1)$, $\|\cdot\|_{W^{s,p}(\rn)}$ is usually called the
\emph{Gagliardo semi-norm}.
It is well known that, for any $s\in[1,\fz)$,
$W^{s,p}(\rn)$ contains only the functions that are almost everywhere constant;
see, for instance, \cite{b02,bvy21} and their references.
Also, the same triviality holds true for $s\in(-\fz,0]$.

\begin{proposition}\label{s=0}
Let $p\in[1,\fz)$ and $s\in(-\fz,0]$.
Then $W^{s,p}(\rn)=\{0\}$.
\end{proposition}

Proposition \ref{s=0} might be well known.
However, to the best of our knowledge,
we did not find a complete proof in the existed literature.
For the convenience of the reader, we present the details here.
To prove Proposition \ref{s=0}, we establish the following lemma.
In what follows, for any $x$, $h\in\rn$ and $t\in(0,\fz)$, let
\begin{align*}
B_t^+:=\lf\{x=(x_1,\dots,x_n)\in\rn:\,|x|\le t\
{\rm and}\
x_j\ge0\ {\rm for\ any}\ j\in\{1,\dots,n\}\r\},
\end{align*}
$$\Delta_h f(x):=f(x+h)-f(x),\quad
\omega(f,t):=\sup_{h\in B_t^+}\|\Delta_h f\|_{L^p(\rn)}^p,$$
and
$$\omega^\ast(f,t):=\frac1{t^n}
\int_{B_{2t}^+\setminus B_t^+}\|\Delta_h f\|_{L^p(\rn)}^p\,dh.$$
\begin{lemma}\label{w-w*}
Let $p\in[1,\fz)$ and $f$ be a measurable function on $\rn$.
Then there exists a positive constant $C$,
depending only on $n$ and $p$,
such that, for any $t\in(0,\fz)$,
$$\omega(f,t)\le C\omega^\ast(f,t).$$
\end{lemma}
\begin{proof}
Let $p$ and $f$ be as in the present lemma.
Notice that, for any $h,\ x,\ s\in\rn$,
\begin{align*}
\Delta_h f(x)
&=f(x+h)-f(x)\\
&=\lf[f(x+h+s)-f(x)\r]-\lf[f(x+h+s)-f(x+h)\r]\\
&=\Delta_{h+s} f(x)-\Delta_s f(x+h),
\end{align*}
which implies that
$$|\Delta_h f(x)|\leq|\Delta_{h+s}f(x)|+|\Delta_{s}f(x)|.$$
Thus, for any $h\in B_{t/2}^+$ and $s\in\rn$, we have
$$\|\Delta_h f\|_{L^p(\rn)}^p
\le \|\Delta_{h+s} f\|_{L^p(\rn)}^p+\|\Delta_s f\|_{L^p(\rn)}^p,$$
which implies that
\begin{align*}
\int_{B_{\frac{3}{2}t}^+\setminus B_t^+}
\|\Delta_h f\|_{L^p(\rn)}^p\,ds
\le \int_{B_{\frac{3}{2}t}^+\setminus B_t^+}
\|\Delta_{h+s} f\|_{L^p(\rn)}^p\,ds
+\int_{B_{\frac{3}{2}t}^+\setminus B_t^+}
\|\Delta_s f\|_{L^p(\rn)}^p\,ds,
\end{align*}
and hence
\begin{align}\label{2t/t}
\|\Delta_h f\|_{L^p(\rn)}^p
&\ls t^{-n}\int_{B_{\frac{3}{2}t}^+\setminus B_t^+}
\|\Delta_{h+s} f\|_{L^p(\rn)}^p\,ds
+t^{-n}\int_{B_{\frac{3}{2}t}^+\setminus B_t^+}
\|\Delta_s f\|_{L^p(\rn)}^p\,ds \noz\\
&\ls t^{-n}\int_{B_{2t}^+\setminus B_t^+}
\|\Delta_{s} f\|_{L^p(\rn)}^p\,ds
\sim\omega^\ast(f,t).
\end{align}
Moreover, observe that, for any $t\in(0,\fz)$,
\begin{align*}
\omega(f,t)&=\sup_{h\in B_t^+}\int_{\rn}|f(x+h)-f(x)|^p\,dx\\
&\ls\sup_{h\in B_t^+}\lf[\int_{\rn}\lf|f(x+h)-f\lf(x+\frac h2\r)\r|^p\,dx
+\int_{\rn}\lf|f\lf(x+\frac h2\r)-f(x)\r|^p\,dx\r]\\
&\ls\sup_{\widetilde{h}\in B_{t/2}^+}\int_{\rn}|f(x+\widetilde{h})-f(x)|^p\,dx
\sim\omega(f,\frac{t}{2}).
\end{align*}
By this and \eqref{2t/t}, we conclude that, for any $t\in(0,\fz)$,
\begin{align*}
\omega(f,t)\ls\omega(f,\frac{t}{2})
\sim \sup_{h\in B_{t/2}^+}\|\Delta_h f\|_{L^p(\rn)}^p
\ls \omega^\ast(f,t),
\end{align*}
which completes the proof of Lemma \ref{w-w*}.
\end{proof}

Next, we prove Proposition \ref{s=0} via Lemma \ref{w-w*}.
\begin{proof}[Proof of Proposition \ref{s=0}]
Let $p\in[1,\fz)$, $s\in(-\fz,0]$, and $f\in W^{s,p}(\rn)$.
Then
\begin{align*}
\fz
&>\int_{\rn}\int_{\rn}\frac{|f(x)-f(y)|^p}{|x-y|^{n+sp}}\,dx\,dy
\ge \int_{B_\fz^+}\int_{\rn}\frac{|f(x+h)-f(x)|^p}{|h|^{n+sp}}\,dx\,dh\\
&=\int_{B_\fz^+}\frac{\|\Delta_h f\|_{L^p(\rn)}^p}{|h|^{n+sp}}\,dh
\ge\sum_{j=1}^\fz\int_{B_{2^{j+1}}^+\setminus B_{2^{j}}^+}
\frac{\|\Delta_h f\|_{L^p(\rn)}^p}{|h|^{n+sp}}\,dh\\
&\sim\sum_{j=1}^\fz2^{-j(n+sp)}
\int_{B_{2^{j+1}}^+\setminus B_{2^{j}}^+}\|\Delta_h f\|_{L^p(\rn)}^p\,dh
\sim\sum_{j=1}^\fz2^{-jsp}\omega^\ast(f,2^j).
\end{align*}
From this and the fact $\lim_{j\to\fz}2^{-jsp}\gtrsim1$,
we deduce that
$$\lim_{j\to\fz}\omega^\ast(f,2^j)=0.$$
By this, the observation that $\omega(f,t)$ is monotonic on $t$,
and Lemma \ref{w-w*}, we obtain, for any $t\in(0,\fz)$,
\begin{align*}
\omega(f,t)\le\lim_{j\to\fz} \omega(f,2^j)
\ls\lim_{j\to\fz} \omega^\ast(f,2^j)=0,
\end{align*}
which implies that, for any $h\in B_\fz^+
:=\{x=(x_1,\dots,x_n)\in\rn:\
x_j\ge0\ {\rm for\ any}\ j\in\{1,\dots,n\}\}$,
$$0=\|\Delta_h f\|_{L^p(\rn)}^p=\int_{\rn}|f(x+h)-f(x)|^p\,dx.$$
Therefore,
\begin{align}\label{B+=0}
0=\int_{B_\fz^+}\int_{\rn}|f(x+h)-f(x)|^p\,dx\,dh.
\end{align}

Next, we consider the case $h\in{\rn\setminus{B_\fz^{+}}}$.
For any $j\in\{1,\dots,n\}$, let
$$
\tau_j:\
\begin{cases}
\qquad\quad\rn&\longrightarrow\qquad\quad\rn\\
(x_1,\dots,x_j,\dots,x_n)&\longmapsto(x_1,\dots,-x_j,\dots,x_n).
\end{cases}
$$
Notice that each octant $\widetilde{B_\fz^+}$ of $\rn$ can be mapped onto
$B_\fz^+$ via a composition of reflections, namely,
there exists a
$\tau_0:=\tau_{j_1}\circ\cdots\circ\tau_{j_k}$
with $\{j_1,\dots,j_k\}\subset\{1,\dots,n\}$
such that $\tau_0(\widetilde{B_\fz^+})=B_\fz^+$,
which implies that
\begin{align*}
&\int_{\widetilde{B_\fz^+}}\int_{\rn}|f(x+h)-f(x)|^p\,dx\,dh
=\int_{B_\fz^+}\int_{\rn}|f(x+h)-f(x)|^p\,dx\,dh.
\end{align*}
Combining this with \eqref{B+=0}, we obtain
\begin{align*}
0=\int_{\rn}\int_{\rn}|f(x+h)-f(x)|^p\,dx\,dh
=\int_{\rn}\int_{\rn}|f(x)-f(y)|^p\,dx\,dy.
\end{align*}
This shows that $f$ equals to some constant almost everywhere on $\rn$,
which completes the proof of Proposition \ref{s=0}.
\end{proof}

Now,
for the endpoint space $JNQ^{n(\frac1q-\frac1p)}_{p,q}(\rn)$,
we have the following conclusion.

\begin{proposition}\label{prop-Wsp}
Let $p,\ q\in[1,\fz)$ and $\az_0:=n(\frac1q-\frac1p)$.
\begin{enumerate}
\item[\rm (i)] If $\az_0\in(0,\fz)$, then
$JNQ^{\az_0}_{p,q}(\rn)\supset W^{\az_0,q}(\rn)$ and
$\|\cdot\|_{JNQ^{\az_0}_{p,q}(\rn)}\le\|\cdot\|_{W^{\az_0,q}(\rn)}$.

\item[\rm (ii)] If $\az_0=0$, then
$JNQ^{\az_0}_{p,q}(\rn)=W^{\az_0,q}(\rn)=\{0\}$.
\end{enumerate}
\end{proposition}
\begin{proof}
We first prove (i).
If $\az_0\in(0,\fz)$, then $\frac1q-\frac1p>0$
and hence $\frac{p}{q}>1$.
Therefore, for any measurable function $f$ on $\rn$,
\begin{align*}
\|f\|_{JNQ^{\az_0}_{p,q}(\rn)}
&=\sup_{\gfz{\ell\in(0,\fz)}{\{Q_j \}_{j}\in\Pi_\ell}}
\lf\{\sum_j \lf[\int_{Q_j}\int_{Q_j}\frac{|f(x)-f(y)|^q}{|x-y|^{n+q\az_0}}
\,dy\,dx \r]^{\frac pq} \r\}^{\frac1p}\\
&\le\sup_{\gfz{\ell\in(0,\fz)}{\{Q_j \}_{j}\in\Pi_\ell}}
\lf[\sum_j \int_{Q_j}\int_{Q_j}\frac{|f(x)-f(y)|^q}{|x-y|^{n+q\az_0}}
\,dy\,dx  \r]^{\frac1q}\\
&=\lf[\int_{\rn}\int_{\rn}\frac{|f(x)-f(y)|^q}{|x-y|^{n+q\az_0}}
\,dy\,dx  \r]^{\frac1q}
=\|f\|_{W^{\az_0,q}(\rn)}.
\end{align*}
This shows that $W^{\az_0,q}(\rn)\subset JNQ^{\az_0}_{p,q}(\rn)$ and
$\|\cdot\|_{JNQ^{\az_0}_{p,q}(\rn)}\le\|\cdot\|_{W^{\az_0,q}(\rn)}$.

Next, we prove (ii). If $\az_0=0$, then $p=q$ and hence,
for any measurable function $f$ on $\rn$,
\begin{align*}
\|f\|_{JNQ^{\az_0}_{p,q}(\rn)}
&=\sup_{\gfz{\ell\in(0,\fz)}{\{Q_j \}_{j}\in\Pi_\ell}}
\lf[\sum_j \int_{Q_j}\int_{Q_j}\frac{|f(x)-f(y)|^p}{|x-y|^{n}}
\,dy\,dx \r]^{\frac1p}\\
&=\lf[\int_{\rn}\int_{\rn}\frac{|f(x)-f(y)|^p}{|x-y|^{n}}
\,dy\,dx  \r]^{\frac1p}
=\|f\|_{W^{\az_0,q}(\rn)}.
\end{align*}
This shows that $JNQ^{\az_0}_{p,q}(\rn)=W^{\az_0,q}(\rn)$
which, together with Proposition \ref{s=0}, further implies that
$JNQ^{\az_0}_{p,q}(\rn)=\{0\}$.
This finishes the proof of Proposition \ref{prop-Wsp}.
\end{proof}

Also, recall that the \emph{Sobolev space $\dot{W}^{1,\gz}(\rn)$}
with $\gz\in(0,\fz)$ is defined to be the set of
all weakly differentiable functions $f$ on $\rn$ such that
$$\|f\|_{\dot{W}^{1,\gz}(\rn)}:=
\lf[\int_{\rn}|\nabla f(x)|^{\gz}\,dx\r]^\frac1\gz
<\fz.$$
Borrowing some ideas from \cite[p.\,8, Theorem 1.4(ii)]{x19},
we establish the following continuously embedding from
$\dot{W}^{1,\gz}(\rn)$ to $JNQ^{\az}_{p,q}(\rn)$.
\begin{proposition}\label{prop-W1r}
Let $\az\in(-\fz,1)$, $p,\ q\in[1,\fz)$, and $\gz\in[q,\fz)$
with $\frac1\gz=\frac1p+\frac1n$.
Then $\dot{W}^{1,\gz}(\rn)\subset JNQ^{\az}_{p,q}(\rn)$
and there exists a positive
constant $C$ such that, for any $f\in\dot{W}^{1,\gz}(\rn)$,
$$\|f\|_{JNQ^{\az}_{p,q}(\rn)}\le C \|f\|_{\dot{W}^{1,\gz}(\rn)}.$$
\end{proposition}
\begin{proof}
Let all the symbols be as in the present proposition.
We claim that, for any given cube $Q$ of $\rn$ with edge length $\ell_0$,
\begin{align}\label{nabla1}
\lf[\int_Q\int_Q\frac{|f(x)-f(y)|^q}{|x-y|^{n+q\az}}
\,dx\,dy\r]^\frac1q
\ls|Q|^{\frac{1-\az}{n}}\lf[\int_{(1+2\sqrt n) Q}|\nabla f(w)|^q
\,dw\r]^\frac1q.
\end{align}
Indeed, for any $y$, $z\in\rn$,
from the fundamental theorem of calculus,
it follows that
$$f(y+z)-f(y)=\int_{0}^{1}z\cdot\nabla f(y+tz)\,dt$$
and hence
\begin{align*}
|f(y+z)-f(y)|\le
\int^{1}_{0}|z|\lf|\nabla f(y+tz)\r|\,dt.
\end{align*}
By this, the Minkowski integral inequality,
and $\az\in(-\fz,1)$, we conclude that
\begin{align*}
&\lf[\int_Q\int_Q\frac{|f(x)-f(y)|^q}{|x-y|^{n+q\az}}
\,dx\,dy\r]^\frac1q\\
&\quad\le
\lf\{\int_Q\int_{B(\mathbf{0},\sqrt{n}\ell_0)}\lf[\frac{|f(y+z)-f(y)|}{|z|}\r]^q
|z|^{q(1-\az)-n}\,dz\,dy\r\}^\frac1q\\
&\quad\le
\lf\{\int_Q\int_{B(\mathbf{0},\sqrt{n}\ell_0)}
\lf[\int_{0}^{1}|\nabla f(y+tz)|\,dt\r]^q
|z|^{q(1-\az)-n}\,dz\,dy\r\}^\frac1q\\
&\quad\le
\int_{0}^{1}\lf[
\int_Q\int_{B(\mathbf{0},\sqrt{n}\ell_0)}|\nabla f(y+tz)|^q
|z|^{q(1-\az)-n}\,dz\,dy
\r]^\frac1q\,dt\\
&\quad\le\lf[
\int_{(1+2\sqrt n) Q}
|\nabla f(w)|^q
\int_{B(\mathbf{0},\sqrt{n}\ell_0)}
|z|^{q(1-\az)-n}\,dz\,dw
\r]^\frac1q\\
&\quad\sim
|Q|^{\frac{1-\az}{n}}\lf[\int_{(1+2\sqrt n) Q}|\nabla f(w)|^q
\,dw\r]^\frac1q.
\end{align*}
This shows that \eqref{nabla1} holds true, and hence we have
\begin{align}\label{nabla2}
\lf[|Q|^{\frac{q\az}{n}-1}\int_Q\int_Q\frac{|f(x)-f(y)|^q}{|x-y|^{n+q\az}}
\,dx\,dy\r]^\frac1q
\ls\lf[|Q|^{\frac{q}{n}-1}\int_{(1+\sqrt n)Q}|\nabla f(w)|^q
\,dw\r]^\frac1q.
\end{align}

Now, let $\{Q_j\}_j\in\Pi_\ell$ with $\ell\in(0,\fz)$.
Then a geometrical observation shows that
there exists a positive constant $C_{(n)}$, depending only on $n$,
such that, for any $x\in\rn$,
\begin{align}\label{finitecover}
\sum_{j}\mathbf{1}_{(1+2\sqrt n)Q_j}(x)\le C_{(n)}.
\end{align}
From \eqref{nabla2},
$\frac1p+\frac1n=\frac1\gz\le\frac1q$, the H\"older inequality,
and \eqref{finitecover},
we deduce that
\begin{align*}
&\sum_{j}|Q_j|\lf[|Q_j|^{\frac{q\az}{n}-1}
\int_{Q_j}\int_{Q_j}\frac{|f(x)-f(y)|^q}{|x-y|^{n+q\az}}
\,dx\,dy\r]^\frac{p}q\\
&\quad\ls\sum_{j}\lf[|Q_j|^{\frac{q}p+\frac{q}{n}}
\fint_{(1+\sqrt n)Q_j}|\nabla f(w)|^q
\,dw\r]^\frac{p}q\\
&\quad\ls\sum_{j}\lf[\
\int_{(1+2\sqrt n)Q_j}|\nabla f(w)|^{1/(\frac1p+\frac1n)}
\,dw\r]^{p(\frac1p+\frac1n)}\\
&\quad\ls\lf[\sum_{j}
\int_{(1+2\sqrt n)Q_j}|\nabla f(w)|^{1/(\frac1p+\frac1n)}
\,dw\r]^{p(\frac1p+\frac1n)}
\ls\lf[
\int_{\rn}|\nabla f(w)|^{\frac1\gz}
\,dw\r]^{\frac{p}\gz},
\end{align*}
where the implicit positive constants are
independent of $\ell$, $\{Q_j\}_j$, and $f$.
This implies that
$$\|f\|_{JNQ^{\az}_{p,q}(\rn)}\ls \|f\|_{\dot{W}^{1,\gz}(\rn)}$$
with $\frac1\gz=\frac1p+\frac1n$,
and hence finishes the proof of Proposition \ref{prop-W1r}.
\end{proof}

As a summary of the main results in this section,
we have the following complete classification.
\begin{theorem}\label{thm-clas}
Let $p,\ q\in[1,\fz)$, $\az\in\rr$, and $\az_0:=n(\frac1q-\frac1p)$.
\begin{enumerate}
\item[\rm (i)] If $\az_0<0$, then $JNQ^{\az}_{p,q}(\rn)=\{0\}$.

\item[\rm (ii)] If $\az_0=0$, then
$$JNQ^{\az}_{p,q}(\rn)=
\begin{cases}
JN^{\rm con}_{p,p}(\rn), &\az\in(-\fz,0), \\
\{0\}, &\az\in[0,\fz).
\end{cases}$$

\item[\rm (iii)] If $\az_0\in(0,1)$, then
$$JNQ^{\az}_{p,q}(\rn)
\begin{cases}
=JN^{\rm con}_{p,q}(\rn), &\az\in(-\fz,0), \\
\supset W^{\az_0,q}(\rn), &\az\in[0,\az_0], \\
=\{0\}, &\az\in(\az_0,\fz).
\end{cases}$$

\item[\rm (iv)]
If $\az_0\in[1,\fz)$, then
$$JNQ^{\az}_{p,q}(\rn)
\begin{cases}
=JN^{\rm con}_{p,q}(\rn), &\az\in(-\fz,0), \\
\supset \dot{W}^{1,(\frac1p+\frac1n)^{-1}}(\rn), &\az\in[0,1), \\
=\{0\}, &\az\in[1,\fz).
\end{cases}$$

\end{enumerate}
\end{theorem}
\begin{proof}
First, (i) immediately follows from Corollary \ref{coro-JN=c}
and Theorem \ref{thm-JNQ123}(i).

Next, we prove (ii). Indeed,
Theorem \ref{thm-JNQ123}(iii) shows the case $\az\in(-\fz,0)$,
Proposition \ref{prop-Wsp} (ii) shows the case $\az=0$,
and Theorem \ref{thm-JNQ123}(ii) shows the case $\az\in(0,\fz)$.
This finishes the proof of (ii).

Now, we show (iii). Indeed,
Theorem \ref{thm-JNQ123}(iii) indicates the case $\az<0$;
by Theorem \ref{thm-JNQ123}(i) and Proposition \ref{prop-Wsp}(i),
we have
$$JN^{\az}_{p,q}(\rn)\supset JN^{\az_0}_{p,q}(\rn)
\supset W^{\az_0,q}(\rn)$$
when $\az\in[0,\az_0]$;
Theorem \ref{thm-JNQ123}(ii) shows the case $\az\in(\az_0,\fz)$,
which completes the proof of (iii).

Finally, we prove (iv).
The cases $\az\in(-\fz,0)$ and $\az\in[1,\fz)$ have been
showed, respectively, in (iii) and (ii) of Theorem \ref{thm-JNQ123}.
Moreover, notice that
$$\az_0\ge1
\Longleftrightarrow
\frac1p+\frac1n\le\frac1q.$$
Using this and Proposition \ref{prop-W1r},
we obtain the case  $\az\in[0,1)$.
This finishes the proof of (iv) and hence of Theorem \ref{thm-clas}.
\end{proof}

\begin{remark}
It is still \emph{unclear} whether or not the inclusion in (iii) and (iv)
of Theorem \ref{thm-clas} are proper.
\end{remark}

\section{Mean Oscillations and Dyadic Counterparts}\label{sec-MO&Dya}

In this section, we show that $JNQ$ spaces can be equivalently
characterized by mean oscillations in Subsection \ref{subsec-MO}.
Moreover, we use this characterization to study the
dyadic counterpart of $JNQ$ spaces in Subsection \ref{subsec-Dya}.

\subsection{Characterization via Mean Oscillations}\label{subsec-MO}

Now, we characterize $JNQ$ spaces in terms of mean oscillations
as in \cite[Section 5]{ejpx00}.
By Corollary \ref{thm-clas}(i), we know that
the space $JNQ^{\az}_{p,q}(\rn)$ is trivial when $p<q$,
and hence we only pay attention to the case $p\ge q$.
In what follows, for any $k\in\zz_+$ and any cube $Q$ of $\rn$,
we use $\mathscr{D}_k(Q)$ to denote the dyadic cubes contained in $Q$ of level $k$;
for any given $q\in[1,\fz)$, $\az\in\rr$, measurable function $f$,
and any cube $Q$ of $\rn$, let
\begin{align}\label{Psifqa}
\Psi_{f,q,\az}(Q):=\lf[\sum_{k=0}^\fz 2^{(q\az-n)k}
\sum_{I\in\mathscr{D}_k(Q)}\fint_{I}|f(x)-f_I|^q\,dx\r]^\frac1q,
\end{align}
where we define $\fint_{I}|f(x)-f_I|^q\,dx=\fz$ if $f\mathbf{1}_I\notin L^1(\rn)$
as in Remark \ref{JNQpqa=JNQcon};
moreover, for any $x=(x_1,\dots,x_n)\in\rn$,
let $|x|_\fz:=\max_{i\in\{1,\dots,n\}}|x_i|$.

\begin{theorem}\label{thm-mean-osc}
Let $1\le q\le p<\fz$ and $\az\in\rr$.
Then $f\in JNQ^{\az}_{p,q}(\rn)$ if and only if
$f$ is measurable on $\rn$ and
\begin{align}\label{Psi}
\|f\|_{\widetilde{JNQ}^{\az}_{p,q}(\rn)}
:=\sup_{\gfz{\ell\in(0,\fz)}{\{Q_j \}_{j}\in\Pi_\ell}}
\lf\{\sum_j |Q_j|\lf[\Psi_{f,q,\az}(Q_j)\r]^p\r\}^\frac1p<\fz,
\end{align}
where $\Psi_{f,q,\az}(Q_j)$ is as in \eqref{Psifqa}
with $Q$ replaced by $Q_j$ for any $j$.
Moreover, $\|\cdot\|_{\widetilde{JNQ}^{\az}_{p,q}(\rn)}\sim
\|\cdot\|_{JNQ^{\az}_{p,q}(\rn)}$.
\end{theorem}
\begin{proof}
Let $p,\ q,\ \az,$ and $f$ be as in the present theorem.
We claim that, for any cube $Q$ of $\rn$ with edge length $\ell(Q)\in(0,\fz)$,
\begin{align}\label{Psi1}
\lf[\Psi_{f,q,\az}(Q)\r]^q\ls\ell^{q\az-n}\int_{Q}\int_{Q}
\frac{|f(x)-f(y)|^q}{|x-y|^{n+q\az}}\,dx\,dy
\end{align}
and
\begin{align}\label{Psi2}
\ell^{q\az-n}\int_{Q}\int_{Q}
\frac{|f(x)-f(y)|^q}{|x-y|^{n+q\az}}\,dx\,dy
\ls\lf[\Psi_{f,q,\az}(Q)\r]^q
+\fint_{E_Q}\lf[\Psi_{f,q,\az}(Q+z)\r]^q\,dz,
\end{align}
where $Q+z:=\{x+z:\ x\in Q\}$,
$E_Q:=\{z\in\rn:\ |z|_\fz\le\ell(Q)\}$,
and the implicit positive constants are independent of $Q$.
Indeed, for the case $q=2$, this claim was proved in
\cite[Lemmas 5.3 and 5.4]{ejpx00} (see also
\cite[p.\,11, Theorem 1.5]{x19});
for any $q\in[1,\fz)$,
we present some details here for the sake of completeness.

We first show \eqref{Psi1}.
From \eqref{Psifqa} and Lemma \ref{osc-equ}(i),
it follows that
\begin{align}\label{psi4.3}
[\Psi_{f,q,\az}(Q)]^{q}
&=\sum\limits_{k=0}^{\fz}2^{(q\az-n)k}\sum\limits_{I\in\mathscr{D}_k(Q)}
\fint_{I}|f(x)-f_{I}|^{q}\,dx\notag\\
&\sim\sum\limits_{k=0}^{\fz}2^{(q\az-n)k}\sum\limits_{I\in\mathscr{D}_k(Q)}
\fint_{I}\fint_{I}|f(x)-f(y)|^{q}\,dx\,dy\notag\\
&\sim\sum\limits_{k=0}^{\fz}\sum\limits_{I\in\mathscr{D}_k(Q)}
\frac{2^{(q\az-n)k}}{(2^{-nk}|Q|)^{2}}
\int_{I}\int_{I}|f(x)-f(y)|^q\,dx\,dy\notag\\
&\sim\int_{\rn}\int_{\rn}k_{Q}(x,y)|f(x)-f(y)|^{q}\,dx\,dy,
\end{align}
where
\begin{align}\label{k}
k_{Q}(x,y):=
\sum\limits_{k=0}^{\fz}\sum\limits_{I\in\mathscr{D}_k(Q)}
2^{(q\az+n)k}|Q|^{-2}\mathbf{1}_{I}(x)\mathbf{1}_{I}(y).
\end{align}
Now, to finish the proof of \eqref{Psi1},
we consider the following two cases on $\az$.

\emph{Case} 1) $\az\in(-\frac{n}{q},\fz)$.
In this case, since
\begin{align*}
x,\ y\in I\in \mathscr{D}_k(Q)
&\Longrightarrow
|x-y|\le\sqrt{n}\ell(I)
=\sqrt{n}2^{-k}\ell(Q)\\
&\Longleftrightarrow
2^{k}\le \sqrt{n}\ell(Q)/|x-y|,
\end{align*}
it follows that, for any $x,\ y\in Q$,
\begin{align*}
k_{Q}(x,y)
&=\sum\limits_{k=0}^{\fz}\sum\limits_{I\in\mathscr{D}_k(Q)}
2^{(q\az+n)k}|Q|^{-2}\mathbf{1}_{I}(x)\mathbf{1}_{I}(y)\\
&\le \sum_{\{k\in\zz_+:\ 2^{k}\le \sqrt{n}\ell(Q)/|x-y|\}}
2^{(q\az+n)k}|Q|^{-2}
\ls\lf[\frac{\ell(Q)}{|x-y|}\r]^{q\az+n}|Q|^{-2}.
\end{align*}
Using this, the observation $k_{Q}(x,y)=0$ unless $x,\ y\in Q$,
and \eqref{psi4.3}, we obtain \eqref{Psi1} when $\az\in(-\frac{n}{q},\fz)$.

\emph{Case} 2) $\az\in(-\fz,-\frac{n}{q}]$.
We first claim that, if $\az\in(-\fz,0)$, then
\begin{align}\label{Psi-sim}
\Psi_{f,q,\az}(Q)\sim {\rm MO}_{f,q}(Q),
\end{align}
where ${\rm MO}_{f,q}(Q)$ is as in \eqref{MOfq},
and the positive equivalence constants are independent of
$f$ and $Q$.
Indeed, notice that, for any $I\in\mathscr{D}_1(Q)$,
by the Minkowski inequality and the H\"older inequality, we have
\begin{align*}
\lf[\fint_{I}\lf|f(x)-f_{Q}\r|^{q}\,dx\r]^\frac1q
&\le\lf[\fint_{I}\lf|f(x)-f_{I}\r|^{q}\,dx\r]^\frac1q
+\lf|f_{I}-f_{Q} \r|\\
&\le\lf[\fint_{I}\lf|f(x)-f_{Q}\r|^{q}\,dx\r]^\frac1q
+2\lf|f_{I}-f_{Q} \r|\\
&\le\lf[\fint_{I}\lf|f(x)-f_{Q}\r|^{q}\,dx\r]^\frac1q
+2\fint_{I}\lf|f(x)-f_{Q}\r|\,dx\\
&\le 3\lf[\fint_{I}\lf|f(x)-f_{Q}\r|^{q}\,dx\r]^\frac1q,
\end{align*}
which implies that
$$\lf[\fint_{I}\lf|f(x)-f_{Q}\r|^{q}\,dx\r]^\frac1q
\sim{\rm MO}_{f,q}(I)
+\lf|f_{I}-f_{Q} \r|,$$
and hence
\begin{align*}
\lf[{\rm MO}_{f,q}(Q)\r]^q
&=\frac1{|Q|}\sum_{I\in\mathscr{D}_1(Q)}
\int_{I}\lf|f(x)-f_{Q}\r|^{q}\,dx \\
&\sim2^{-n}\sum_{I\in\mathscr{D}_1(Q)}
\lf\{\lf[{\rm MO}_{f,q}(I)\r]^q
+\lf|f_{I}-f_{Q} \r|^q \r\}
\gtrsim 2^{-n}
\sum_{I\in\mathscr{D}_1(Q)}
\lf[{\rm MO}_{f,q}(I)\r]^q.
\end{align*}
From this and the mathematical induction,
it follows that, for any $k\in\nn$,
\begin{align}\label{MOk}
\lf[{\rm MO}_{f,q}(Q)\r]^q
\gtrsim 2^{-kn}
\sum_{I\in\mathscr{D}_k(Q)}
\lf[{\rm MO}_{f,q}(I)\r]^q.
\end{align}
Moreover, notice that
\begin{align*}
\Psi_{f,q,\az}(Q)
&=\lf[{\rm MO}_{f,q}(Q)\r]^q
+\sum_{k=1}^\fz \sum_{I\in\mathscr{D}_k(Q)}
2^{(q\az-n)k}\lf[{\rm MO}_{f,q}(I)\r]^q.
\end{align*}
This, combined with \eqref{MOk} and $\az\in(-\fz,0)$,
implies that
\begin{align*}
\lf[{\rm MO}_{f,q}(Q)\r]^q
&\le\Psi_{f,q,\az}(Q)
\ls\sum_{k=0}^\fz 2^{q\az k}\lf[{\rm MO}_{f,q}(Q)\r]^q
\sim\lf[{\rm MO}_{f,q}(Q)\r]^q,
\end{align*}
which shows that \eqref{Psi-sim} holds true.
Using \eqref{Psi-sim}, Lemma \ref{osc-equ}(i), and \eqref{az<-1/2},
we obtain \eqref{Psi1} when $\az\in(-\fz,-\frac{n}{q}]$.

Combining Cases 1) and 2), we find that \eqref{Psi1} holds true.

Next, we prove \eqref{Psi2}.
By \eqref{psi4.3} and the Fubini theorem, we obtain
\begin{align*}
&\fint_{E_Q}\lf[\Psi_{f,q,\az}(Q+z)\r]^q\,dz
=
\int_{\rn}\int_{\rn}\fint_{E_Q}
k_{Q+z}(x,y)\,dz\,|f(x)-f(y)|^{q}\,dx\,dy.
\end{align*}
Thus, to prove \eqref{Psi2},
it suffices to show that, for any given $x$, $y\in Q$,
\begin{align}\label{k+fintk}
\frac{\ell^{q\az-n}}{|x-y|^{n+q\az}}
\ls
k_{Q}(x,y)+\fint_{E_Q}k_{Q+z}(x,y)\,dz.
\end{align}
We first consider the case $x$, $y\in Q$ with
$|x-y|_{\fz}\le2^{-1}\ell(Q)$.
In this case, there exists an $m\in\zz_+$ such that
\begin{align}\label{|x-y|sim}
2^{-m-2}\ell(Q)<|x-y|_{\fz}\le2^{-m-1}\ell(Q).
\end{align}
Notice that, if $|z|_{\fz}>\ell(Q)$,
then $x\notin Q+z$ and hence $k_{Q+z}(x,y)=0$.
From this, \eqref{k}, and \eqref{|x-y|sim} we deduce that
\begin{align}\label{k-lowerbdd}
\fint_{E_Q}k_{Q+z}(x,y)\,dz
&=\frac{1}{|E_Q|}\int_{\rn}k_{Q+z}(x,y)\,dz\noz\\
&\gtrsim\frac{1}{|Q|}\int_{\rn}
\sum_{I\in\mathscr{D}_m(Q+z)}2^{(q\az+n)m}|Q|^{-2}
\mathbf{1}_{I}(x)\mathbf{1}_{I}(y)\,dz\noz\\
&\sim [\ell(Q)]^{-3n}
\lf[\frac{\ell(Q)}{|x-y|_\fz}\r]^{q\az+n}
\sum_{I\in\mathscr{D}_m(Q)}\int_{\rn}
\mathbf{1}_{I+z}(x)\mathbf{1}_{I+z}(y)\,dz\noz\\
&\sim\frac{[\ell(Q)]^{q\az-2n}}{|x-y|^{n+q\az}}
\sum_{I\in\mathscr{D}_m(Q)}\int_{\rn}
\mathbf{1}_{I+z}(x)\mathbf{1}_{I+z}(y)\,dz.
\end{align}
Observe that, for any $I\in\mathscr{D}_m(Q)$,
and $z$, $x$, $y\in\rn$,
$$\mathbf{1}_{I+z}(x)\mathbf{1}_{I+z}(y)
=\mathbf{1}_{I-x}(-z)\mathbf{1}_{I-y}(-z),$$
and hence
\begin{align*}
&\sum_{I\in\mathscr{D}_m(Q)}\int_{\rn}
\mathbf{1}_{I+z}(x)\mathbf{1}_{I+z}(y)\,dz\\
&\quad=\sum_{I\in\mathscr{D}_m(Q)}|(I-x)\cap(I-y)|
\ge\sum_{I\in\mathscr{D}_m(Q)}2^{-n}|I|
=2^{-n}|Q|,
\end{align*}
where the inequality holds true because
\eqref{|x-y|sim}\ $\Longrightarrow\ell(I)-|x-y|_\fz\ge2^{-1} \ell(I)
\Longrightarrow$ there exists some cube $\widetilde{I}$
with edge length $\frac12\ell(I)$ such that
$[(I-x)\cap(I-y)]\supset\widetilde{I}$.
By this and \eqref{k-lowerbdd}, we conclude that
\begin{align}\label{k+fintk1}
\fint_{E_Q}k_{Q+z}(x,y)\,dz\gtrsim
\frac{\ell^{q\az-n}}{|x-y|^{n+q\az}}.
\end{align}
We then consider the case $x$, $y\in Q$ with
$|x-y|_{\fz}>2^{-1}\ell(Q)$.
In this case, we have $|x-y|\sim\ell(Q)$ which,
together with \eqref{k} with $k=0$,
further implies that
$$
k_{Q}(x,y)\geq|Q|^{-2}\sim\frac{\ell^{q\az-n}}{|x-y|^{n+q\az}}.
$$
Combining this with \eqref{k+fintk1}, we obtain
\eqref{k+fintk}, and hence \eqref{Psi2} holds true.
Altogether, we complete the proofs of both \eqref{Psi1} and \eqref{Psi2},
and hence of the above claim.

Now, from \eqref{Psi1}, we deduce that, if $f\in JNQ^{\az}_{p,q}(\rn)$,
then
\begin{align*}
\|f\|_{\widetilde{JNQ}^{\az}_{p,q}(\rn)}
&=\sup_{\gfz{\ell\in(0,\fz)}{\{Q_j \}_{j}\in\Pi_\ell}}
\lf\{\sum_j |Q_j|\lf[\Psi_{f,q,\az}(Q_j)\r]^p\r\}^\frac1p\\
&\ls\sup_{\gfz{\ell\in(0,\fz)}{\{Q_j \}_{j}\in\Pi_\ell}}
\lf\{\sum_j |Q_j|\lf[|Q_j|^{\frac{q\az}n-1}\int_{Q_j}\int_{Q_j}
\frac{|f(x)-f(y)|^q}{|x-y|^{n+q\az}}\,dx\,dy\r]^\frac pq\r\}^\frac1p\\
&\sim\|f\|_{JNQ^{\az}_{p,q}(\rn)}<\fz
\end{align*}
and hence \eqref{Psi} holds true.
Conversely, if the measurable function $f$ satisfies \eqref{Psi},
then, by \eqref{Psi2}, $p/q\ge1$, and the Minkowski inequality,
we conclude that
\begin{align*}
\|f\|_{JNQ^{\az}_{p,q}(\rn)}
&=\sup_{\gfz{\ell\in(0,\fz)}{\{Q_j \}_{j}\in\Pi_\ell}}
\lf\{\sum_j |Q_j|\lf[|Q_j|^{\frac{q\az}n-1}\int_{Q_j}\int_{Q_j}
\frac{|f(x)-f(y)|^q}{|x-y|^{n+q\az}}\,dx\,dy\r]^\frac pq\r\}^\frac1p \\
&\ls\sup_{\gfz{\ell\in(0,\fz)}{\{Q_j \}_{j}\in\Pi_\ell}}
\lf\{\sum_j |Q_j|\lf[\Psi_{f,q,\az}(Q_j)\r]^p\r\}^\frac1p \\
&\quad+\sup_{\gfz{\ell\in(0,\fz)}{\{Q_j \}_{j}\in\Pi_\ell}}
\lf[\sum_j |Q_j|\lf\{\fint_{E_Q}
\lf[\Psi_{f,q,\az}(Q_j+z)\r]^q\,dz\r\}^{\frac pq} \r]^\frac1p \\
&\ls \sup_{\gfz{\ell\in(0,\fz)}{\{Q_j \}_{j}\in\Pi_\ell}}
\lf\{\sum_j |Q_j|\lf[\Psi_{f,q,\az}(Q_j)\r]^p\r\}^\frac1p \\
&\quad+\sup_{\gfz{\ell\in(0,\fz)}{\{Q_j \}_{j}\in\Pi_\ell}}
\fint_{E_Q}\lf\{\sum_j |Q_j|
\lf[\Psi_{f,q,\az}(Q_j+z)\r]^p\r\}^\frac1p\,dz  \\
&\ls\sup_{\gfz{\ell\in(0,\fz)}{\{Q_j \}_{j}\in\Pi_\ell}}
\lf\{\sum_j |Q_j|\lf[\Psi_{f,q,\az}(Q_j)\r]^p\r\}^\frac1p
\sim\|f\|_{\widetilde{JNQ}^{\az}_{p,q}(\rn)}<\fz.
\end{align*}
This finishes the proof of Theorem \ref{thm-mean-osc}.
\end{proof}

\subsection{Dyadic $JNQ$ Spaces}\label{subsec-Dya}

Following \cite[Section 7]{ejpx00},
we study the dyadic counterparts of $JNQ$ spaces
in this subsection.
In what follows, for any $k\in\zz$,
$\mathscr{D}_k(\rn)$ denotes the set of all dyadic cubes contained in $\rn$
of level $k$, and
$\mathscr{D}(\rn):=\bigcup_{k\in\zz}\mathscr{D}_k(\rn)$.
In addition, for any $x,\ y\in\rn$,
the \emph{dyadic distance $\delta(x,y)$} is defined by setting
$$\delta(x,y):=\inf\lf\{\ell(I):\ x,\ y\in I \in \mathscr{D}(\rn)\r\}.$$
Notice that the dyadic distance is infinite between points in different octants
and, for any $x,\ y\in Q\in\mathscr{D}(\rn)$,
\begin{align}\label{dzxy}
|x-y|\le\sqrt{n}\dz(x,y)\le\sqrt{n}\ell(Q).
\end{align}

\begin{definition}\label{def-dJNQ}
Let $p,\ q\in[1,\fz)$ and $\az\in\rr$.
The \emph{dyadic $JNQ$ space} $JNQ^{\az,{\rm dyadic}}_{p,q}(\rn)$
is defined to be the set of all measurable functions $f$ on $\rn$
such that
\begin{align*}
&\|f\|_{JNQ^{\az,{\rm dyadic}}_{p,q}(\rn)}
:=\sup_{k\in\zz}\lf\{\sum_{Q\in \mathscr{D}_k(\rn)}\lf|Q\r|
\lf[|Q|^{\frac{q\az}n-1}\int_{Q}\int_{Q}\frac{|f(x)-f(y)|^q}{[\delta(x,y)]^{n+q\az}}
\,dx\,dy \r]^\frac{p}q \r\}^{\frac{1}{p}}<\fz.
\end{align*}
\end{definition}

The following lemma is a slight variant (changing `$2$' into `$q$') of
\cite[Lemma 7.1]{ejpx00}; we omit the details here.

\begin{lemma}\label{Psidyasim}
Let $q\in[1,\fz)$ and $\az\in\rr$. Then there
exists a positive constant $C$ such that,
for any dyadic cube $Q\in\mathscr{D}(\rn)$
and any $f\in L^q(Q)$,
\begin{align*}
C^{-1}\Psi_{f,q,\az}(Q)
&\le
\lf[|Q|^{\frac{q\az}n-1}\int_{Q}\int_{Q}\frac{|f(x)-f(y)|^q}{[\delta(x,y)]^{n+q\az}}
\,dx\,dy \r]^\frac{1}q
\le C\Psi_{f,q,\az}(Q),
\end{align*}
where $\Psi_{f,q,\az}(Q)$ is the same as in \eqref{Psifqa}.
\end{lemma}

As a consequence of Lemma \ref{Psidyasim},
we immediately have the following equivalent norm
on $JNQ^{\az,{\rm dyadic}}_{p,q}(\rn)$.
\begin{proposition}\label{prop-dyadic-norm}
Let $p,\ q\in[1,\fz)$ and $\az\in\rr$.
Then $f\in JNQ^{\az,{\rm dyadic}}_{p,q}(\rn)$ if and only if
$f$ is measurable on $\rn$ and
\begin{align*}
\|f\|_{\widetilde{JNQ}^{\az,{\rm dyadic}}_{p,q}(\rn)}
:=\sup_{k\in\zz} \lf\{\sum_{Q\in \mathscr{D}_k(\rn)}
|Q|\lf[\Psi_{f,q,\az}(Q)\r]^p\r\}^\frac1p<\fz,
\end{align*}
where $\Psi_{f,q,\az}(Q)$ is the same as in \eqref{Psifqa},
Moreover,
$\|\cdot\|_{\widetilde{JNQ}^{\az,{\rm dyadic}}_{p,q}(\rn)}
\sim\|\cdot\|_{JNQ^{\az,{\rm dyadic}}_{p,q}(\rn)}$.
\end{proposition}
\begin{remark}
The following Lemma \ref{Psi=Phi} implies that,
if $p,\ q\in[1,\fz)$ and $\az\in(-\fz,1/q)$,
then
$$\|f\|_{\widetilde{JNQ}^{\az,{\rm dyadic}}_{p,q}(\rn)}
\sim
\sup_{k\in\zz} \lf\{\sum_{Q\in \mathscr{D}_k(\rn)}
|Q|\lf[\Phi_{f,q,\az}(Q)\r]^p\r\}^\frac1p$$
with $\Phi_{f,q,\az}$ be the same as in \eqref{Phifqa},
and the positive equivalence constants independent of $f$.
\end{remark}

Also, we have an analogue of Theorem \ref{thm-JNQ123}
for dyadic $JNQ$ spaces.
\begin{theorem}\label{thm-d123}
Let $p,\ q\in[1,\fz)$. Then dyadic $JNQ$ spaces
enjoy the following properties.
\begin{enumerate}
\item[\rm (i)] (Decreasing in $\alpha$)
If $-\fz<\az_1\le\az_2<\fz$, then
$JNQ^{\az_1,{\rm dyadic}}_{p,q}(\rn)\supset
JNQ^{\az_2,{\rm dyadic}}_{p,q}(\rn)$.

\item[\rm (ii)] (Triviality for large $\az$)
If $\az\in(n(\frac1q-\frac1p),\fz)$, then
$JNQ^{\az,{\rm dyadic}}_{p,q}(\rn)$
contains only functions that are almost everywhere
constant in each octant.

\item[\rm (iii)] (Triviality for negative $\az$)
If $\az_1,\ \az_2\in(-\fz,0)$,
then $JNQ^{\az_1,{\rm dyadic}}_{p,q}(\rn)
=JNQ^{\az_2,{\rm dyadic}}_{p,q}(\rn)$
with equivalent norms.
\end{enumerate}
\end{theorem}
\begin{proof}
First, using \eqref{dzxy},
we find that, for any $\az_1,\ \az_2\in\rr$
with $\az_1\le\az_2$, any $f\in JNQ^{\az_2,{\rm dyadic}}_{p,q}(\rn)$,
and any dyadic cube $Q\in\mathscr{D}(\rn)$,
\begin{align*}
&|Q|^{\frac{q\az_1}{n}-1}\int_{Q}\int_{Q}
\frac{|f(x)-f(y)|^q}{[\delta(x,y)]^{n+q\az_1}}
\,dx\,dy\\
&\quad=|Q|^{\frac{q\az_1}{n}-1}\int_{Q}\int_{Q}
\frac{|f(x)-f(y)|^q}{[\delta(x,y)]^{n+q\az_2}}
\lf[\delta(x,y)\r]^{q(\az_2-\az_1)}\,dx\,dy\\
&\quad\leq [\ell(Q)]^{q\az_1-n}\int_{Q}\int_{Q}
\frac{|f(x)-f(y)|^q}{[\delta(x,y)]^{n+q\az_2}}
[\ell(Q)]^{q(\az_2-\az_1)}\,dx\,dy\\
&\quad=|Q|^{\frac{q\az_2}{n}-1}\int_{Q}\int_{Q}
\frac{|f(x)-f(y)|^q}{[\delta(x,y)]^{n+q\az_2}}
\,dx\,dy,
\end{align*}
which implies that
$
\|f\|_{JNQ^{\az_1,{\rm dyadic}}_{p,q}(\rn)}
\leq
\|f\|_{JNQ^{\az_2,{\rm dyadic}}_{p,q}(\rn)}
$
and hence
$$JNQ^{\az_1,{\rm dyadic}}_{p,q}(\rn)
\supset JNQ^{\az_2,{\rm dyadic}}_{p,q}(\rn).$$
This shows (i) of the present theorem.

Next, we prove (ii) of the present theorem.
Let $\az\in(n(\frac1q-\frac1p),\fz)$,
$f\in JNQ^{\az,{\rm dyadic}}_{p,q}(\rn)$,
and $Q\in\mathscr{D}(\rn)$.
Suppose that $Q_m$ is a dyadic cube satisfying $Q_m\supset Q$
and $\ell(Q_m)=2^m \ell(Q)$.
From Proposition \ref{prop-dyadic-norm} and \eqref{Psifqa},
it follows that
\begin{align*}
\fz&>\|f\|_{JNQ^{\az,{\rm dyadic}}_{p,q}(\rn)}
\ge|Q_m|^{\frac1p}\Psi_{f,q,\az}(Q_m)
\geq
2^{[\az-n(\frac1q-\frac1p)]m}
\lf[\fint_{Q}|f(x)-f_Q|^q\,dx\r]^\frac1q.
\end{align*}
Letting $m\to\fz$, we obtain $\fint_{Q}|f(x)-f_Q|^q\,dx=0$,
which implies that $f$ is a constant almost everywhere in each octant,
and hence (ii) holds true.

Finally, (iii) follows from
Proposition \ref{prop-dyadic-norm} and
\eqref{Psi-sim},
which completes the proof of Theorem \ref{thm-d123}.
\end{proof}

\begin{remark}\label{rem-dSharp}
Let $p,\ q\in[1,\fz)$ and $\az_0:=n(\frac1q-\frac1p)$.
Then, corresponding to Proposition \ref{prop-Wsp},
we claim that
$JNQ^{\az_0,{\rm dyadic}}_{p,q}(\rn)\supset W^{\az_0,q}(\rn)$
if
$\az_0\ge0$.
Indeed, by \eqref{Psi1} and the observation
$\az_0\ge0\Longleftrightarrow p/q\ge1$,
we obtain
\begin{align*}
&\sup_{k\in\zz} \lf\{\sum_{Q\in \mathscr{D}_k(\rn)}
|Q|\lf[\Psi_{f,q,\az_0}(Q)\r]^p\r\}^\frac1p\\
&\quad\ls\sup_{k\in\zz}\lf[\sum_{Q\in \mathscr{D}_k(\rn)}
|Q|\lf\{[\ell(Q)]^{q\az_0-n}
\int_{Q}\int_{Q}
\frac{|f(x)-f(y)|^q}{|x-y|^{n+q\az_0}}\,dx\,dy\r\}^
{\frac{p}{q}}\r]^\frac1p\\
&\quad\sim\sup_{k\in\zz}\lf[\sum_{Q\in \mathscr{D}_k(\rn)}
\lf\{\int_{Q}\int_{Q}
\frac{|f(x)-f(y)|^q}{|x-y|^{n+q\az_0}}\,dx\,dy\r\}^
{\frac{p}{q}}\r]^\frac1p \\
&\quad\ls\sup_{k\in\zz}\lf[\sum_{Q\in \mathscr{D}_k(\rn)}
\int_{Q}\int_{Q}\frac{|f(x)-f(y)|^q}{|x-y|^{n+q\az_0}}\,dx\,dy\r]^\frac1q\\
&\quad\sim\lf[\int_{\rn}\int_{\rn}
\frac{|f(x)-f(y)|^q}{|x-y|^{n+q\az_0}}\,dx\,dy\r]^\frac1q,
\end{align*}
which, together with Proposition \ref{prop-dyadic-norm},
further implies that
$\|f\|_{JNQ^{\az_0,{\rm dyadic}}_{p,q}(\rn)}
\ls\|f\|_{W^{\az_0,q}(\rn)}$.
This shows the above claim,
and hence, if $\az_0\in(0,1)$, then
\begin{align}\label{neq0}
JNQ^{\az_0,{\rm dyadic}}_{p,q}(\rn)
\neq \{{\rm a.\,e.\ constant\ in\ each\ octant}\}.
\end{align}
Moreover, similarly to the proof of Theorem \ref{thm-clas},
we find that \eqref{neq0} is false for any $\az_0\in(-\fz,0]$.
However, when $\az_0\in[1,\fz)$,
it is interesting to ask whether or not
\eqref{neq0} still holds true,
which is still \emph{unclear} so far.
\end{remark}

Ess\'en et al. in \cite[Theorem 7.9]{ejpx00} used $\Psi_{f,q,\az}(Q)$
to establish the following relation between the space $Q$
and its dyadic analogue $Q_{\az}^{\rm d}(\rn)$, that is, for any $\az\in(-\fz,1/2)$,
$$Q_{\az}(\rn)=\lf[Q_{\az}^{\rm d}(\rn)\cap{\rm BMO}(\rn)\r].$$
Now, we establish a corresponding result for the space $JNQ$
as follows.
Notice that we only need to consider the case $p\ge q$,
otherwise $p<q$ and hence both $JNQ^{\az}_{p,q}(\rn)$ and $JN^{\rm con}_{p,q}(\rn)$
are trivial due to Theorem \ref{thm-clas}(i) and Corollary \ref{coro-JN=c}.

\begin{theorem}\label{thm-dyadic}
Let $1\le q\le p<\fz$ and $\az\in(-\fz,1/q)$.
Then
$$JNQ^{\az}_{p,q}(\rn)=
\lf[JNQ^{\az,{\rm dyadic}}_{p,q}(\rn)\cap JN^{\rm con}_{p,q}(\rn)\r].$$
\end{theorem}
To prove Theorem \ref{thm-dyadic},
we need the following geometrical lemma
(Lemma \ref{lem-partition} below)
which is a refinement of both \cite[Lemma 2.4]{tyyACV}
and \cite[Lemma 2.5]{jtyyz21}.
In what follows, for any set $A$ of $\rn$,
we use $\overline{A}$ to denote its \emph{closure};
moreover, two sets $A$ and $B$ are said to be \emph{mutually adjacent}
if $\overline{A}\cap\overline{B}\neq\emptyset$.

\begin{definition}\label{def-domi}
Let $k\in\zz$ and $Q$ be a cube of $\rn$ with edge length
$\ell\in(2^{-k-1},2^{-k}]$.
Observe that there exist mutually adjacent dyadic cubes
$\{Q^{(j)}\}_{j=1}^{2^n}\subset\mathscr{D}_k(\rn)$,
with $Q^{(1)}$ being the left and lower one in
$\{Q^{(j)}\}_{j=1}^{2^n}$,
such that
$Q\cap Q^{(1)}\neq\emptyset$
and
$Q\subset\bigcup_{j=1}^{2^n}Q^{(j)}$;
see the figure below for two examples of cubes when $n=2$.
Then the \emph{dominated cube} of $Q$, denoted by $Q^\flat$,
is defined by setting $Q^\flat:=\bigcup_{j=1}^{2^n}Q^{(j)}$.
Moreover, $Q^{(1)}$ is called the {\em lefter and lower cube} of $Q$.
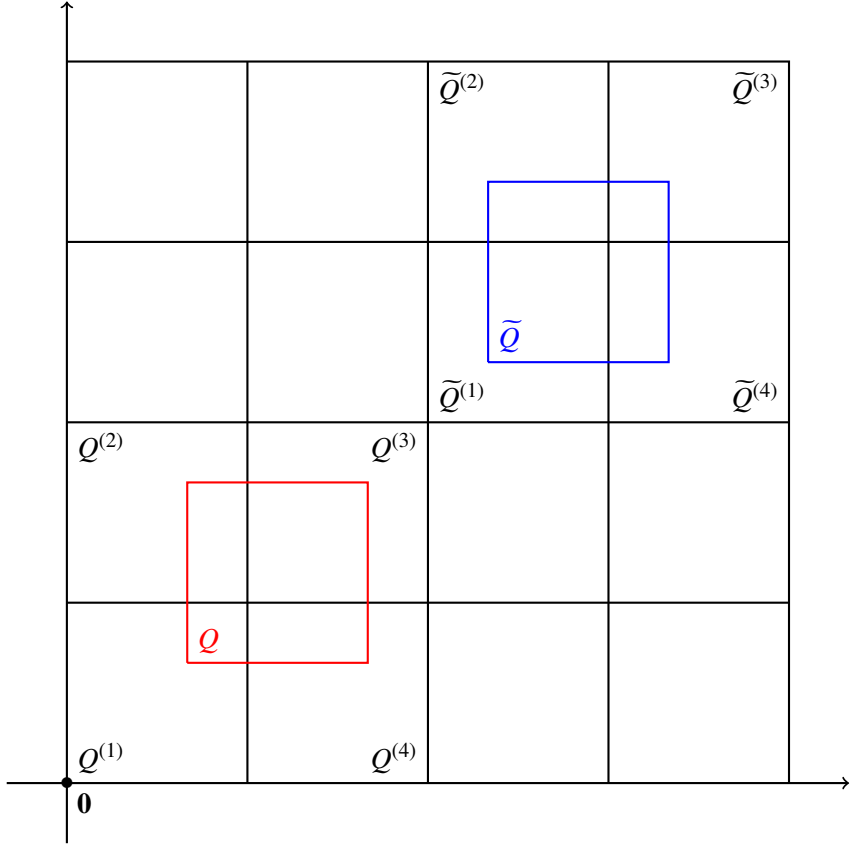
\begin{center}
\begin{tikzpicture} [scale = 1.6]
\draw [thick, ->=5pt] (2.5,3)--(9.5,3);
\draw [thick, ->=5pt] (3,2.5)--(3,9.5);

\draw[thick] (3,9)--(9,9)--(9,3);
\draw[thick] (3,4.5)--(9,4.5);
\draw[thick] (4.5,3)--(4.5,9);
\draw[thick] (3,6)--(9,6);
\draw[thick] (6,3)--(6,9);
\draw[thick] (3,7.5)--(9,7.5);
\draw[thick] (7.5,3)--(7.5,9);
{\red\draw[thick] (4,4)--(4,5.5)--(5.5,5.5)--(5.5,4)--(4,4);}
{\blue\draw[thick] (6.5,6.5)--(6.5,8)--(8,8)--(8,6.5)--(6.5,6.5);}

\node [above right] at (3,3) {$Q^{(1)}$};
\node [below right] at (3,6) {$Q^{(2)}$};
\node [below left] at (6,6) {$Q^{(3)}$};
\node [above left] at (6,3) {$Q^{(4)}$};
{\red \node [above right] at (4,4) {$Q$};}

\node [above right] at (6,6) {$\widetilde{Q}^{(1)}$};
\node [below right] at (6,9) {$\widetilde{Q}^{(2)}$};
\node [below left] at (9,9) {$\widetilde{Q}^{(3)}$};
\node [above left] at (9,6) {$\widetilde{Q}^{(4)}$};
{\blue \node [above right] at (6.5,6.5) {$\widetilde{Q}$};}

\node [below right] at (3,3) {$\mathbf{0}$};
\node at (3,3) {$\bullet$};
\end{tikzpicture}
\end{center}
\end{definition}

\begin{lemma}\label{lem-partition}
Let $k\in\zz$ and $\{Q_j\}_{j\in\nn}\in\Pi_\ell$
with $\ell\in(2^{-k-1},2^{-k}]$.
Then $\{Q_j\}_{j\in\nn}$ can be divided into
$6^n$ subfamilies $\{\cq_m\}_{m=1}^{6^n}$
(may be empty for some $m$) such that
\begin{enumerate}
\item[\rm (i)] $\{Q_j\}_{j\in\nn}=\bigcup_{m=1}^{6^n}\cq_m$;

\item[\rm (ii)] for any $m\in\{1,\dots,6^n\}$,
all cubes in $\{Q^\flat:\ Q\in\cq_m\}$ are mutually disjoint,
where $Q^\flat$ is as in Definition \ref{def-domi}.
\end{enumerate}
\end{lemma}

\begin{proof}
Let all the symbols be as in the present lemma.
Divide $\mathscr{D}_k(\rn)$ into $3^n$ sparse subfamilies
as follows:
\begin{align*}
\mathscr{S}_1:=&\lf\{[0,2^{-k})^n+3\cdot2^{-k}\mathbf{j}:\
\mathbf{j}\in\zz^n\r\},\\
\mathscr{S}_2:=&\mathscr{S}_1+(2^{-k},0,\ldots,0),\\
\mathscr{S}_3:=&\mathscr{S}_1+(2\cdot2^{-k},0,\ldots,0),\\
\vdots\\
\mathscr{S}_{3^n}:=&\mathscr{S}_1+(2\cdot2^{-k},2\cdot2^{-k},\ldots,2\cdot2^{-k}).
\end{align*}
By some geometrical observations, we find that,
if $Q_1\in\mathscr{S}_k$ and $Q_2\in\mathscr{S}_l$ with $k\neq l$,
then
\begin{align}\label{empty0}
\overline{Q_1}\cap\overline{Q_2}=\emptyset;
\end{align}
moreover, for any cubes $Q_k,\ Q_l\in\{Q_j\}_{j\in\nn}$,
\begin{align}\label{empty1}
Q_k^\flat\cap Q_l^\flat=\emptyset
\Longleftrightarrow
\overline{Q_k^{(1)}}\cap\overline{Q_l^{(1)}}=\emptyset
\end{align}
(see, for instance, the cubes $Q$ and $\widetilde{Q}$ in the figure
of Definition \ref{def-domi}).
Furthermore, from $\{Q_j\}_{j\in\nn}\in\Pi_\ell$
and $\ell\in(2^{-k-1},2^{-k}]$, it follows that
each dyadic cube in $\mathscr{D}_k(\rn)$
may be the lefter and lower cube (see Definition \ref{def-domi})
of at most $2^n$ cubes in $\{Q_j\}_{j\in\nn}$.
Combining this, \eqref{empty0}, and \eqref{empty1},
we obtain $3^n\times2^n=6^n$ desired subfamilies,
which completes the proof of Lemma \ref{lem-partition}.
\end{proof}

The following two lemmas are, respectively, the variants of
\cite[Lemmas 5.7 and 5.8]{ejpx00}.
Indeed, Lemma \ref{Psi-AI} shows that
$\Psi_{f,q,\az}$ is almost increasing (see Definition \ref{def-ai}),
and Lemma \ref{Psi=Phi} is the reverse of \eqref{Psi1};
we omit their proofs here because,
as in the proof of \eqref{Psi1} in Theorem \ref{thm-mean-osc},
it suffices to change `$2$' into `$q$' and
make some corresponding modifications.
\begin{lemma}\label{Psi-AI}
Let $q\in[1,\fz)$, $\az\in(-\fz,1/q)$,
and $f$ be a measurable function on $\rn$.
Then there exists a positive constant $C$,
depending only on $n$, $q$, and $\az$, such that,
for any cubes $Q_1$, $Q_2$ with $Q_1\subset Q_2$
and $\ell(Q_1)=\frac12\ell(Q_2)$,
$$\Psi_{f,q,\az}(Q_1)\le C \Psi_{f,q,\az}(Q_2),$$
where $\Psi_{f,q,\az}$ is the same as in \eqref{Psifqa}.
\end{lemma}

\begin{lemma}\label{Psi=Phi}
Let $q\in[1,\fz)$, $\az\in(-\fz,1/q)$, and $f\in L^q(Q)$.
Let $\Phi_{f,q,\az}$ and $\Psi_{f,q,\az}$ be, respectively,
as in \eqref{Phifqa} and \eqref{Psifqa}.
Then
$$\Psi_{f,q,\az}(Q)
\sim
\Phi_{f,q,\az}(Q)$$
with the positive equivalence constants depending only on
$n$, $q$, and $\az$.
\end{lemma}

Now, we prove Theorem \ref{thm-dyadic}.
\begin{proof}[Proof of Theorem \ref{thm-dyadic}]
Let $1\le q\le p<\fz$ and $\az\in\rr$.
It is obvious to find that
$$\|\cdot\|_{JNQ^{\az,{\rm dyadic}}_{p,q}(\rn)}
\le\|\cdot\|_{JNQ^{\az}_{p,q}(\rn)}$$
and hence
$$JNQ^{\az}_{p,q}(\rn)\subset
JNQ^{\az,{\rm dyadic}}_{p,q}(\rn),$$
which, combined with Theorem \ref{thm-JNQ123},
further implies that
$$JNQ^{\az}_{p,q}(\rn)\subset
\lf[JNQ^{\az,{\rm dyadic}}_{p,q}(\rn)\cap JN^{\rm con}_{p,q}(\rn)\r].$$

Conversely, let $\az\in(-\fz,1/q)$,
$f\in JNQ^{\az,{\rm dyadic}}_{p,q}(\rn)\cap JN^{\rm con}_{p,q}(\rn)$,
and $\{Q_j\}_{j\in\nn}\in\Pi_\ell$ with $\ell\in[2^{-k-1},2^{-k})$
for some given $k\in\zz$.
For any $j\in\nn$, let $Q_j^\flat=\bigcup_{i=1}^{2^n}Q_j^{(i)}$
be the dominated cube of $Q_j$ as in Definition \ref{def-domi}.
Also, let $\{\cq_m\}_{m=1}^{6^n}$ be as in Lemma \ref{lem-partition}.
Then, by the observation $|Q_j^\flat|\sim|Q_j|$,
Lemmas \ref{Psi=Phi} and \ref{lem-partition},
\eqref{Psifqa}, and Proposition \ref{prop-dyadic-norm},
we conclude that
\begin{align*}
&\sum_{j}\lf|Q_{j}\r|
\lf[|Q_j|^{\frac{q\az}n-1}\int_{Q_j}\int_{Q_j}
\frac{|f(x)-f(y)|^q}{|x-y|^{n+q\az}}\,dy\,dx \r]^\frac{p}q\\
&\quad\ls\sum_{j}\lf|Q_{j}^\flat\r|
\lf[\lf|Q_j^\flat\r|^{\frac{q\az}n-1}\int_{Q_j^\flat}\int_{Q_j^\flat}
\frac{|f(x)-f(y)|^q}{|x-y|^{n+q\az}}\,dy\,dx \r]^\frac{p}q\\
&\quad\sim\sum_j \lf|Q_j^\flat\r|
\lf[\Psi_{f,q,\az}(Q_j^\flat)\r]^p\\
&\quad\sim\sum_{m=1}^{6^n}\sum_{\{j:\ Q_j\in\cq_m\}}
\lf|Q_j^\flat\r|
\lf[\Psi_{f,q,\az}\lf(Q_j^{(1)}\cup\cdots\cup Q_j^{(2^n)}\r)\r]^p\\
&\quad\sim\sum_{m=1}^{6^n}\sum_{\{j:\ Q_j\in\cq_m\}}
\lf|Q_j^\flat\r|
\lf\{\fint_{Q_j^\flat}\lf|f(x)-f_{Q_j^\flat}\r|^q\,dx
+\lf[\sum_{i=1}^{2^n}\Psi_{f,q,\az}(Q_j^{(i)})
\r]^\frac1q\r\}^p\\
&\quad\ls\sum_{m=1}^{6^n}
\sum_{\{j:\ Q_j\in\cq_m\}}
\lf|Q_j^\flat\r|
\lf[\fint_{Q_j^\flat}\lf|f(x)-f_{Q_j^\flat}\r|^q
\,dx\r]^\frac pq
+\sum_{m=1}^{6^n}\sum_{Q\in\mathscr{D}_k}
|Q|\lf[\Psi_{f,q,\az}(Q)\r]^p\\
&\quad\ls\|f\|_{JN^{\rm con}_{p,q}(\rn)}^p
+\|f\|_{JNQ^{\az,{\rm dyadic}}_{p,q}(\rn)}^p,
\end{align*}
where the implicit positive constants depend only on
$n,\ p,\ q,$ and $\az$.
This implies that $f\in JNQ^{\az}_{p,q}(\rn)$ and
$\|f\|_{JNQ^{\az}_{p,q}(\rn)}
\ls\|f\|_{JN^{\rm con}_{p,q}(\rn)}
+\|f\|_{JNQ^{\az,{\rm dyadic}}_{p,q}(\rn)}$,
and hence
$$JNQ^{\az}_{p,q}(\rn)\supset
\lf[JNQ^{\az,{\rm dyadic}}_{p,q}(\rn)\cap JN^{\rm con}_{p,q}(\rn)\r],$$
which completes the proof of Theorem \ref{thm-dyadic}.
\end{proof}

At the end of this section,
we discuss some other dyadic $JNQ$-type norms.
Let $\Phi_{f,q,\az}$ be the same as in \eqref{Phifqa},
$$\|f\|_{JNQ^{\az,{\rm I}}_{p,q}(\rn)}
:=\sup \lf\{\sum_j |Q_j|\lf[\Phi_{f,q,\az}(Q_j)\r]^p\r\}^\frac1p$$
with the supremum taken over all collections $\{Q_j\}_j$ of
subcubes of $\rn$ with pairwise disjoint interiors
and same dyadic edge length,
$$\|f\|_{JNQ^{\az,{\rm II}}_{p,q}(\rn)}
:=\sup_{k\in\zz}\lf\{\sum_{Q\in \mathscr{D}_k(\rn)}\lf|Q\r|
\lf[\Phi_{f,q,\az}(Q_j)\r]^p \r\}^{\frac{1}{p}},$$
and
$$\|f\|_{JNQ^{\az,{\rm III}}_{p,q}(\rn)}
:=\sup \lf\{\sum_j |Q_j|\lf[\Phi_{f,q,\az}(Q_j)\r]^p\r\}^\frac1p$$
with the supremum taken over all collections $\{Q_j\}_j$ of disjoint
dyadic cubes of $\rn$ ($\{Q_j\}_j$ may have different edge lengths).
\begin{remark}\label{rem-dJNQs}
\begin{enumerate}
\item[\rm (i)]
From Proposition \ref{prop-integral} and
the proof of Corollary \ref{coro-IntJNQ},
we easily deduce that
$$\|\cdot\|_{JNQ^{\az,{\rm I}}_{p,q}(\rn)}
\sim\|\cdot\|_{\widetilde{JNQ}^{\az}_{p,q}(\rn)}
\sim\|\cdot\|_{JNQ^{\az}_{p,q}(\rn)}.$$
Moreover, we apparently have
$$\|\cdot\|_{JNQ^{\az,{\rm I}}_{p,q}(\rn)}
\ge\|\cdot\|_{JNQ^{\az,{\rm II}}_{p,q}(\rn)}$$
via their definitions,
but the reverse inequality is obviously false
(it suffices to consider functions which are
almost everywhere constant in each octant,
but not equal to some constant almost everywhere on $\rn$).

\item[\rm (ii)]
From Lemmas \ref{Psidyasim} and \ref{Psi=Phi},
it follows that, for any $\az\in(-\fz,\frac1q)$,
$$\|\cdot\|_{JNQ^{\az,{\rm dyadic}}_{p,q}(\rn)}
\sim \|\cdot\|_{JNQ^{\az,{\rm II}}_{p,q}(\rn)};$$
but the case $\az\in[\frac1q,\fz)$ is still \emph{unclear} so far.
The main obstacle is that we do not know whether or not
Lemma \ref{Psi=Phi} still holds true for $\az\in[\frac1q,\fz)$;
see also \cite[p.\,22, Section 1.5]{x19} for some similar
problems in $Q$ spaces.

\item[\rm (iii)]
Recall that Berkovits et al. in \cite[Corollary 3.3]{bkm16}
showed the John--Nirenberg inequality of $JN_p^{\rm dyadic}(Q_0)$
whose norm is taken over all collections of disjoint dyadic cubes of $Q_0$
(not necessary to have equal edge length).
Moreover, Yue and Dafni \cite{yd09} established the John--Nirenberg inequality of
$Q_\az(\rn)$ with $\az\in[0,1)$;
see also \cite[p.\,32, Theorem 2.3]{x19} and \cite{y10}.
Thus, it is very  interesting to find the John--Nirenberg inequality
of $JNQ^{\az,{\rm III}}_{p,1}(\rn)$ with $p\in[1,\fz)$ and $\az\in\rr$.
This is a challenging problem which is still \emph{open} so far.
\end{enumerate}
\end{remark}

\section{Composition Operators on $JNQ$ Spaces}\label{sec-co}

This section is devoted to the left and the right
composition operators on $JNQ$ spaces.
Let $L:\ \rr\to\rr$ and $R:\ \rn\to\rn$ be
suitable mappings.
Then, for any measurable function $f$ and any $x\in\rn$,
the \emph{left composition operator} $\mathscr{C}_{L}$
is defined by setting
$$\mathscr{C}_{L}(f)(x)=L(f(x)),$$
and the \emph{right composition operator} $\mathscr{C}_{R}$
is defined by setting
$$\mathscr{C}_{R}(f)(x)=f(R(x)).$$

We first show the boundedness of $\mathscr{C}_{L}$
in Theorem \ref{thm-LC} below, which is an application of
Proposition \ref{prop-W1r}.
Recall that the \emph{Lipschitz space} ${\rm Lip}(\rn)$
is defined to be the set of all
measurable functions $f$ on $\rn$ such that
$$\|f\|_{{\rm Lip}(\rn)}:=\sup\lf\{\frac{|f(x)-f(y)|}{|x-y|}:\
x,\ y\in\rn\ {\rm and}\ x\neq y\r\}<\fz,$$
and a well-known result of Campanato \cite{c64} indicates that,
for any $q\in[1,\fz)$,
\begin{align}\label{lip=osc}
\|f\|_{{\rm Lip}(\rn)}\sim
\sup_{{\rm cube\ }Q\subset\rn}
|Q|^{-\frac1n}\lf[\fint_Q\lf|f(x)-f_Q\r|^q\,dx\r]^\frac1q
\end{align}
with the positive equivalence constants independent of $f$.

\begin{theorem}\label{thm-LC}
Let $\az\in(-\fz,1)$, $p,\ q\in[1,\fz)$, $\gz\in[q,\fz)$
with $\frac1\gz=\frac1p+\frac1n$,
and $L:\ \rr\to\rr$ be a continuous mapping.
Then
\begin{enumerate}
\item[\rm (i)]
$\mathscr{C}_{L}$ is bounded on $JNQ^{\az}_{p,q}(\rn)$
if and only if $L\in {\rm Lip}(\rr)$;

\item[\rm (ii)]
$f\in JNQ^{\az}_{p,q}(\rn)$ if and only if
there exists a $g\in JNQ^{\az}_{p,q}(\rn)$
such that
$$\frac1g\in JNQ^{\az}_{p,q}(\rn)
\quad{\rm and}\quad
f=g-\frac1g.$$
\end{enumerate}
\end{theorem}
\begin{proof}
Let all the symbols be as in the present proposition.
We first show (i).
If $L\in{\rm Lip}(\rr)$,
then, for any $x,\ y\in\rn$, we have
$$\lf|L(f(x))-L(f(y))\r|
\le\|L\|_{{\rm Lip}(\rr)}\lf|f(x)-f(y)\r|$$
and hence
$$\|\mathscr{C}_{L}(f)\|_{JNQ^{\az}_{p,q}(\rn)}\le
\|L\|_{{\rm Lip}(\rr)}\|f\|_{JNQ^{\az}_{p,q}(\rn)}.$$
Conversely, let $\mathscr{C}_{L}$ be bounded on $JNQ^{\az}_{p,q}(\rn)$.
To obtain $L\in{\rm Lip}(\rr)$,
via \eqref{lip=osc}, it suffices to show that,
for any given interval $I\subset\rr$ with finite length $\ell$,
\begin{align}\label{lip-1}
|I|^{-1}\lf[\fint_I\lf|L(t)-L_I\r|^q\,dt\r]^\frac1q\ls1.
\end{align}
To this end, let $Q:=I^n:=I\times\cdots\times I$ be a cube of $\rn$,
and
$$f_1(x):=
\begin{cases}
x_1, & x=(x_1,\dots,x_n)\in Q, \\
0, & x\in\rn\setminus Q.
\end{cases}$$
Then, from Proposition \ref{prop-W1r}, it follows that
\begin{align}\label{f1-JNQ}
\lf\|f_1\r\|_{JNQ^{\az}_{p,q}(\rn)}
\ls\lf\|f_1\r\|_{\dot{W}^{1,\gz}(\rn)}
\sim\lf[\int_{Q}|(1,0,\dots,0)|^{\gz}\,dx\r]^\frac1\gz
\sim |Q|^{\frac1p+\frac1n}.
\end{align}
Notice that
\begin{align*}
\mathscr{C}_{L}(f_1)(x)
&=L(f_1(x))
=
\begin{cases}
L(x_1), & x=(x_1,\dots,x_n)\in Q, \\
0, & x\in\rn\setminus Q
\end{cases}
\end{align*}
and hence
\begin{align}\label{ItoQ}
\fint_I\lf|L(t)-L_I\r|^q\,dt
=\fint_Q\lf|\mathscr{C}_{L}(f_1)(x)-
[\mathscr{C}_{L}(f_1)]_Q\r|^q\,dx.
\end{align}
Moreover, by Definition \ref{def-JNconpqa},
Theorem \ref{thm-JNQ123}, the boundedness of
$\mathscr{C}_{L}$, and \eqref{f1-JNQ}, we conclude that
\begin{align*}
&\lf\{|Q|
\lf[\fint_{Q}\lf|\mathscr{C}_{L}(f_1)(x)-
[\mathscr{C}_{L}(f_1)]_Q\r|^{q}\,dx\r]^{\frac{p}{q}}\r\}^{\frac{1}{p}}\\
&\quad\leq
\lf\|\mathscr{C}_{L}(f_1)\r\|_{JN^{\rm con}_{p,q}(\rn)}
\ls
\lf\|\mathscr{C}_{L}(f_1)\r\|_{JNQ^{\az}_{p,q}(\rn)}
\ls
\lf\|f_1\r\|_{JNQ^{\az}_{p,q}(\rn)}
\ls |Q|^{\frac1p+\frac1n},
\end{align*}
which, together with \eqref{ItoQ} and the fact that
$|Q|=|I|^n$, further implies that
$$
|I|^{-1}\lf[\fint_I\lf|L(x)-L_I\r|^q\,dx\r]^\frac1q
\ls
|I|^{-1}|Q|^{\frac1n}
\sim 1.
$$
and hence \eqref{lip-1} holds true.
This shows that $\phi\in{\rm Lip}(\rn)$
and hence finishes the proof of (i).

We then show (ii).
If there exists a $g\in JNQ^{\az}_{p,q}(\rn)$
such that $\frac1g\in JNQ^{\az}_{p,q}(\rn)$ and $f=g-\frac1g$,
then we apparently have $f\in JNQ^{\az}_{p,q}(\rn)$.
Conversely, let $f\in JNQ^{\az}_{p,q}(\rn)$.
For any $x\in\rn$, let
$$g(x):=\frac{\sqrt{x^2+2^2}+x}2.$$
Then, for any $x\in\rn$,
$$\frac{1}{g(x)}=\frac{\sqrt{x^2+2^2}-x}2,$$
and hence
\begin{align}\label{1=g-g}
x=g(x)-\frac{1}{g(x)}.
\end{align}
Notice that $g,\ \frac{1}{g}\in {\rm Lip}(\rn)$.
Then, from (i) and $f\in JNQ^{\az}_{p,q}(\rn)$, it follows that
$g(f),\ \frac1{g(f)}\in JNQ^{\az}_{p,q}(\rn)$,
which, combined with \eqref{1=g-g}, further implies that
\begin{align*}
f=g(f)-\frac1{g(f)}.
\end{align*}
This finishes the proof of (ii) and hence of Theorem \ref{thm-LC}.
\end{proof}

\begin{remark}
We have the following observation.
\begin{enumerate}
\item[\rm (i)]
For more studies of left composition operators on
Lebesgue spaces, Sobolev spaces,
and function spaces of Besov--Triebel--Lizorkin type,
we refer the reader to \cite{bms20,bms14,bms10} and the monograph \cite{rs96};

\item[\rm (ii)]
for more studies of left composition operators on
$\BMO$ and $Q$ spaces, we refer the reader to
\cite{bls02,x17} and the monograph \cite{x19}.
\end{enumerate}
\end{remark}

Next, we consider the right composition operator.
Recall that, if $R$ is a quasicomformal mapping
with some particular geometrical assumptions,
then Reimann \cite[Theorem 2]{r74} showed that $\mathscr{C}_{R}$ is bounded on
$\BMO(\rn)$,
and Koskela et al. \cite[Theorem 1.3]{kxzz17} showed that
$\mathscr{C}_{R}$ is also bounded on $Q_\az(\rn)$ with $\az\in(0,1)$.
However, the following proposition indicates that
the quasicomformal mapping may be an unsuitable
object to study the boundedness of right composition operators
on $JNQ$ spaces.
In what follows, for any $r\in(0,\fz)$
and any cube $Q$ of $\rn$, let
$r\times Q:=\{rx:\ x\in Q\}$.

\begin{proposition}\label{prop-NoQCM}
Let $p\in[1,\fz)$, $r\in(0,\fz)$, and
$$\mathscr{C}_{R}:=
\begin{cases}
\qquad\rn&\longrightarrow\qquad\rn\\
x=(x_1,\dots,x_n)&\longmapsto rx=(rx_1,\dots,rx_n).
\end{cases}$$
Then
\begin{align}\label{-n/p}
\|\mathscr{C}_{R}(f)\|_{JN_p^{\rm con}(\rn)}=
r^{-n/p}\|f\|_{JN_p^{\rm con}(\rn)}.
\end{align}
\end{proposition}
\begin{proof}
Let all the symbols be as in the present proposition.
Also, let $\{Q_j\}_j\in\Pi_\ell$ with $\ell\in(0,\fz)$.
Then, for any $j$, we have
$$\fint_{Q_j}\mathscr{C}_{R}(f)(y)\,dy
=\fint_{Q_j}f(ry)\,dy
=\fint_{r\times Q_j}f(z)\,dz$$
and hence $(\mathscr{C}_{R}(f))_{Q_j}=f_{r\times Q_j}$.
By this, we have
\begin{align*}
&\sum_{j}\lf|Q_{j}\r|\lf[\fint_{Q_{j}}
\lf|\mathscr{C}_{R}(f)(x)-\lf(\mathscr{C}_{R}(f)\r)_{Q_j}\r|
\,dx\r]^p\\
&\quad=\sum_{j}\lf|Q_{j}\r|\lf[\fint_{Q_{j}}
\lf|f(rx)-f_{r\times Q_j}\r|
\,dx\r]^p
=\sum_{j}\lf|Q_{j}\r|\lf[\fint_{r\times Q_j}
\lf|f(u)-f_{r\times Q_j}\r|
\,du\r]^p\\
&\quad=r^{-n}\sum_{j}\lf|r\times Q_j\r|\lf[\fint_{r\times Q_j}
\lf|f(u)-f_{r\times Q_j}\r|
\,du\r]^p.
\end{align*}
which, together with the observation
$\{r\times Q_j\}_j\in\Pi_{r\ell}$,
further implies that \eqref{-n/p} holds true.
This finishes the proof of Proposition \ref{prop-NoQCM}.
\end{proof}
\begin{remark}\label{rem-RCO}
Notice that $JN_\fz^{\rm con}(\rn)=\BMO(\rn)$;
see Remark \ref{rem-moC}(ii) or \cite[Proposition 2.11]{jtyyz21}.
Also, observe that $\mathscr{C}_{R}$ in Proposition \ref{prop-NoQCM}
is the most elementary quasiconformal mapping.
Letting $r\to0^+$ in \eqref{-n/p}, we find that $p=\fz$ is
the only possible $p$ such that right composition operators,
generated by quasiconformal mappings,
are bounded on $JN_p^{\rm con}(\rn)$.
Therefore, it is interesting to find suitable conditions
of $R$ such that $\mathscr{C}_{R}$ is bounded on $JNQ$ spaces,
which is still \emph{unclear} so far.
\end{remark}

\noindent\textbf{Acknowledgements}.
The authors would like to thank Professor Feng Dai
for some useful discussions on Proposition \ref{s=0},
and also to thank Professor Yuan Zhou
for some useful discussions on Proposition \ref{prop-NoQCM}
as well as the topic of quasiconformal mappings.

\bigskip

\noindent Jin Tao, Zhenyu Yang  and Wen Yuan (Corresponding author)

\medskip

\noindent Laboratory of Mathematics and Complex Systems
(Ministry of Education of China),
School of Mathematical Sciences, Beijing Normal University,
Beijing 100875, People's Republic of China

\smallskip

\noindent{\it E-mails:} \texttt{jintao@mail.bnu.edu.cn} (J. Tao)

\noindent\phantom{{\it E-mails:}} \texttt{zhenyuyang@mail.bnu.edu.cn} (Z. Yang)

\noindent\phantom{{\it E-mails:}} \texttt{wenyuan@bnu.edu.cn} (W. Yuan)


\begin{thebibliography}{99}

\bibitem{bkm16}
L. Berkovits, J. Kinnunen and J. M. Martell,
Oscillation estimates, self-improving results and good-$\lambda$ inequalities,
J. Funct. Anal. 270 (2016), 3559-3590.

\vspace{-0.3cm}

\bibitem{bdlwy19}
J. J. Betancor, X.-T. Duong, J. Li, B. D. Wick and D. Yang,
Product Hardy, BMO spaces and iterated commutators associated
with Bessel Schr\"odinger operators,
Indiana Univ. Math. J. 68 (2019), 247-289.

\vspace{-0.3cm}

\bibitem{bls02}
G. Bourdaud, M. Lanza de Cristoforis and W. Sickel,
Functional calculus on BMO and related spaces,
J. Funct. Anal. 189 (2002), 515-538.

\vspace{-0.3cm}

\bibitem{bms20}
G. Bourdaud, M. Moussai and W. Sickel,
A necessary condition for composition in Besov spaces,
Complex Var. Elliptic Equ. 65 (2020), 22-39.

\vspace{-0.3cm}

\bibitem{bms14}
G. Bourdaud, M. Moussai and W. Sickel,
Composition operators acting on Besov spaces on the real line,
Ann. Mat. Pura Appl. (4) 193 (2014), 1519-1554.

\vspace{-0.3cm}

\bibitem{bms10}
G. Bourdaud, M. Moussai and W. Sickel,
Composition operators on Lizorkin--Triebel spaces,
J. Funct. Anal. 259 (2010), 1098-1128.

\vspace{-0.3cm}

\bibitem{bbm15}
J. Bourgain, H. Brezis and P. Mironescu,
A new function space and applications,
J. Eur. Math. Soc. (JEMS) 17 (2015), 2083-2101.

\vspace{-0.3cm}

\bibitem{b02}
H. Brezis,
How to recognize constant functions. A connection with Sobolev spaces,
Uspekhi Mat. Nauk 57 (2002), 59-74; translation in
Russian Math. Surveys 57 (2002), 693-708.

\vspace{-0.3cm}

\bibitem{bvy21}
H. Brezis, J. Van Schaftingen and P.-L. Yung,
A surprising formula for Sobolev norms,
Proc. Natl. Acad. Sci. USA 118 (2021), Paper No. e2025254118, 6 pp.

\vspace{-0.3cm}

\bibitem{c64}
S. Campanato,
Propriet\`a di una famiglia di spazi funzionali,
Ann. Scuola Norm. Sup. Pisa Cl. Sci. (3) 18 (1964), 137-160.

\vspace{-0.3cm}

\bibitem{cdlsy21}
P. Chen, X.-T. Duong, J. Li, L. Song and L. Yan,
BMO spaces associated to operators with generalised
Poisson bounds on non-doubling manifolds with ends,
J. Differential Equations 270 (2021), 114-184.

\vspace{-0.3cm}

\bibitem{cdlsy17}
P. Chen, X.-T. Duong, J. Li, L. Song and L. Yan,
Carleson measures, BMO spaces and balayages associated
to Schr\"odinger operators,
Sci. China Math. 60 (2017), 2077-2092.

\vspace{-0.3cm}

\bibitem{dhky18}
G. Dafni, T. Hyt\"onen, R. Korte and H. Yue,
The space {$JN_p$}: nontriviality and duality,
J. Funct. Anal. 275 (2018), 577-603.

\vspace{-0.3cm}

\bibitem{dllw19}
X.-T. Duong, H. Li, J. Li and B. D. Wick,
Lower bound of Riesz transform kernels and commutator theorems
on stratified nilpotent Lie groups,
J. Math. Pures Appl. (9) 124 (2019), 273-299.

\vspace{-0.3cm}


\bibitem{dlsvwy21}
X.-T. Duong, J. Li, E. Sawyer, M. N. Vempati, B. D. Wick and D. Yang,
A two weight inequality for Calder\'on--Zygmund operators on
spaces of homogeneous type with applications,
J. Funct. Anal. 281 (2021), Paper No. 109190, 65 pp.

\vspace{-0.3cm}

\bibitem{dlwy21}
X.-T. Duong, J. Li, B. D. Wick and D. Yang,
Characterizations of product Hardy spaces in Bessel setting,
J. Fourier Anal. Appl. 27 (2021), Paper No. 24, 65 pp.

\vspace{-0.3cm}

\bibitem{dy05}
X.-T. Duong and L. Yan,
Duality of Hardy and BMO spaces associated
with operators with heat kernel bounds,
J. Amer. Math. Soc. 18 (2005), 943-973.

\vspace{-0.3cm}

\bibitem{ejpx00}
M. Ess\'en, S. Janson, L. Peng  and J. Xiao,
$Q$ spaces of several real variables,
Indiana Univ. Math. J. 49 (2000), 575-615.

\vspace{-0.3cm}

\bibitem{jtyyz21}
H. Jia, J. Tao, W. Yuan, D. Yang and Y. Zhang,
Special John--Nirenberg--Campanato
spaces via congruent cubes,
Sci. China Math. https://doi.org/10.1007/s11425-021-1866-4.

\vspace{-0.3cm}

\bibitem{jtyyzS1}
H. Jia, J. Tao, W. Yuan, D. Yang and Y. Zhang,
Boundedness of Calder\'on--Zygmund operators on special
John--Nirenberg--Campanato and	Hardy-type spaces via congruent cubes,
Anal. Math. Phys. (2021), https://doi.org/10.1007/s13324-021-00626-w.

\vspace{-0.3cm}

\bibitem{jtyyzS2}
H. Jia, J. Tao, W. Yuan, D. Yang and Y. Zhang,
Boundedness of fractional integrals on special
John--Nirenberg--Campanato and Hardy-type spaces via congruent cubes,
Submitted.

\vspace{-0.3cm}

\bibitem{jyyzLP}
H. Jia, W. Yuan, D. Yang and Y. Zhang,
Estimates for Littlewood--Paley operators on
special John--Nirenberg--Campanato spaces via congruent cubes,
Submitted.

\vspace{-0.3cm}

\bibitem{jn61}
F. John and L. Nirenberg,
On functions of bounded mean oscillation,
Comm. Pure Appl. Math. 14 (1961), 415-426.

\vspace{-0.3cm}

\bibitem{kxzz17}
P. Koskela, J. Xiao, Y. Zhang and Y. Zhou,
A quasiconformal composition problem for the $Q$-spaces,
J. Eur. Math. Soc. (JEMS) 19 (2017), 1159-1187.

\vspace{-0.3cm}

\bibitem{lw17}
J. Li and B. D. Wick,
Characterizations of $H^1_{\Delta N}(\rn)$ and ${\rm BMO}_{\Delta N}(\rn)$
via weak factorizations and commutators,
J. Funct. Anal. 272 (2017), 5384-5416.

\vspace{-0.3cm}

\bibitem{r74}
H. M. Reimann,
Functions of bounded mean oscillation and quasiconformal mappings,
Comment. Math. Helv. 49 (1974), 260-276.

\vspace{-0.3cm}

\bibitem{rs96}
T. Runst and W. Sickel,
Sobolev Spaces of Fractional Order, Nemytskij Operators,
and Nonlinear Partial Differential Equations,
De Gruyter Series in Nonlinear Analysis and Applications 3,
Walter de Gruyter \& Co., Berlin, 1996.

\vspace{-0.3cm}

\bibitem{txyyJFAA}
J. Tao, Q. Xue, D. Yang and W. Yuan,
XMO and weighted compact bilinear commutators,
J. Fourier Anal. Appl. 27 (2021), no. 3, Paper No. 60, 34 pp.

\vspace{-0.3cm}

\bibitem{tyyM}
J. Tao, D. Yang and W. Yuan,
A survey on function spaces of John--Nirenberg type,
Mathematics 9 (2021), 2264, https://doi.org/10.3390/math9182264.

\vspace{-0.3cm}

\bibitem{tyyACV}
J. Tao, D. Yang and W. Yuan,
Vanishing John--Nirenberg spaces,
Adv. Calc. Var. (2021),
https://doi.org/10.1515/acv-2020-0061.

\vspace{-0.3cm}

\bibitem{tyyNA}
J. Tao, D. Yang and W. Yuan,
John--Nirenberg--Campanato spaces,
Nonlinear Anal. 189 (2019), 111584, 36 pp.
\vspace{-0.3cm}

\bibitem{x19}
J. Xiao,
$Q_\alpha$ Analysis on Euclidean Spaces,
Advances in Analysis and Geometry 1, De Gruyter, Berlin, 2019.

\vspace{-0.3cm}

\bibitem{x17}
J. Xiao,
The transport equation in the scaling invariant Besov
or Ess\'en--Janson--Peng--Xiao space,
J. Differential Equations 266 (2019), 7124-7151.

\vspace{-0.3cm}

\bibitem{ycyy20}
S. Yang, D.-C. Chang, D. Yang and W. Yuan,
Weighted gradient estimates for elliptic problems with
Neumann boundary conditions in Lipschitz and (semi-)convex domains,
J. Differential Equations 268 (2020), 2510-2550.

\vspace{-0.3cm}

\bibitem{y10}
H. Yue,
A fractal function related to the John--Nirenberg inequality for $Q_\az(\rn)$,
Canad. J. Math. 62 (2010), 1182-1200.

\vspace{-0.3cm}

\bibitem{yd09}
H. Yue and G. Dafni,
A John--Nirenberg type inequality for $Q_\az(\rn)$,
J. Math. Anal. Appl. 351 (2009), 428-439.

\end{thebibliography}
\end{document}